\documentclass[10pt]{article}

\usepackage[english]{babel}
\usepackage[utf8]{inputenc}
\usepackage[T1]{fontenc}

\usepackage{amsmath, amsthm, amssymb, mathtools}

\usepackage[shortlabels]{enumitem}

\usepackage{cases}

\newcommand\imp[1]{\textit{#1}}
\newcommand\EC{\textbf{(EC)}\xspace}
\newcommand\NC{\textbf{(NC)}\xspace}

\usepackage[
backend=biber,
firstinits=true,
style=alphabetic,
sorting=nyt,
isbn=false,
url=false,
doi=true,
maxbibnames=99
]{biblatex}
\AtBeginDocument{\def\MR#1{}}

\usepackage{hyperref}

\usepackage{sectsty}

\usepackage{tikz}
\usetikzlibrary{positioning}
\usepackage{caption}

\usepackage[a4paper]{geometry}

\providecommand{\keywords}[1]{\textbf{Keywords.} #1}
\providecommand{\MSC}[1]{\textbf{2010 Mathematics Subject Classification.} #1}

\newtheorem{theorem}{Theorem}[section]
\newtheorem{corollary}[theorem]{Corollary}
\newtheorem{lemma}[theorem]{Lemma}
\newtheorem{proposition}[theorem]{Proposition}
\newtheorem{example}[theorem]{Example}

\theoremstyle{definition}
\newtheorem{definition}[theorem]{Definition}
\newtheorem{remark}[theorem]{Remark}

\renewcommand\epsilon{\varepsilon}
\renewcommand\mapsto{\longmapsto}

\usepackage{color}

\usepackage{xspace}

\newcommand{\R}{\field{R}\xspace}
\newcommand{\C}{\field{C}\xspace}
\newcommand{\N}{\field{N}\xspace}

\newcommand{\field}[1]{\ensuremath{\mathbb{#1}}}

\newcommand{\ens}[1]{ \left\{#1\right\} }

\newcommand\diag{\mathrm{diag} \,}

\newcommand{\TinfNC}{T_{\mathrm{inf}}^{\scriptscriptstyle \NC}}
\newcommand{\TinfEC}{T_{\mathrm{inf}}^{\scriptscriptstyle \EC}}

\newcommand\TCN{T_{\scriptscriptstyle \mathrm{[CN]} }}

\newcommand\pt[1]{\frac{\partial #1}{\partial t}}
\newcommand\px[1]{\frac{\partial #1}{\partial x}}

\newcommand\pxi[1]{\frac{\partial #1}{\partial \xi}}

\newcommand\Tau{\mathcal{T}}

\newcommand\chir{\hat{\chi}}
\newcommand\ssinr{\hat{s}^{\mathrm{in}}}

\newcommand{\rank}{\mathrm{rank} \,}

\newcommand\Id{\mathrm{Id}}

\newcommand{\lin}[1]{\mathcal{L}(#1)}

\def\norm#1{\left\|#1\right\|}
\newcommand\abs[1]{\left|#1\right|}
\newcommand\st{\quad \middle| \quad}

\newcommand{\ps}[3]{ {\left\langle #1 , #2 \right\rangle}_{#3} }
\newcommand\ddt{\frac{d}{dt}}

\newcommand\setCN{\mathcal{B}}

\newcommand\opA[1]{A_{\scriptscriptstyle #1}}
\newcommand\opS[1]{S_{\scriptscriptstyle #1}}

\newcommand\dds{\frac{d}{ds}}

\newcommand\Xm[2]{\ell_{#1}\left[#2\right]}

\newcommand\xmax{x_0}

\newcommand\tr{{\mathsf{T}}}

\newcommand\Qcan{Q^c}

\newcommand\Qadj{\hat{Q}}

\newcommand\Gadj{\hat{G}}

\newcommand\rhoCN{\theta}
\newcommand\rhomax{\bar{\rho}}

\title{Control of 1D first-order linear hyperbolic systems: a brief survey}

\author{
Long Hu\thanks{School of Mathematics, Shandong University, Jinan, Shandong 250100, China.  E-mail: \texttt{hul@sdu.edu.cn}}
\and
Guillaume Olive\thanks{Faculty of Mathematics and Computer Science, Jagiellonian University, ul. {\L}ojasiewicza 6, 30-348 Krak\'{o}w, Poland. E-mail: \texttt{math.golive@gmail.com} or \texttt{guillaume.olive@uj.edu.pl}}
}

\date{}

\begin{document}

\maketitle

\begin{abstract}
At the occasion of the 70th birthday anniversary of Jean-Michel Coron, we review the literature on the control of 1D first-order linear hyperbolic systems, a topic where he has made significant contributions.
\end{abstract}

\keywords{Hyperbolic systems; Minimal control time; $LU$ decomposition; Compactness-uniqueness method; Backstepping method; Equivalent systems; Stabilization}

\vspace{0.2cm}
\MSC{35L40; 93B05; 93D15; 45D05}

\tableofcontents

\section{Introduction}

Many physical models are described by one-dimensional (1D) hyperbolic systems, such as navigable rivers, irrigation channels, heat exchangers, tubular plug flow chemical reactors, gas pipelines, chromatography and road traffic.
The control problem for such systems plays an important role both from theoretical and practical point of view.
Roughly speaking, the control problem for evolution equations discusses the possibility of driving within a finite time the state of a system to a prescribed final target state by means of inputs, called \imp{controls}, that can be applied either at the boundary or inside the domain where the system evolves.
These problems are commonly studied in science and engineering, especially for hyperbolic systems. For example, controlling the Saint-Venant equations is crucial for the regulation of navigable rivers in engineering applications.
A concrete example is the Meuse river in Wallonia, southern Belgium (see \cite[Chapter 8]{BC16}), where boundary controls are implemented to ensure sufficient water flow for navigation while simultaneously preventing flooding caused by high river discharge.
Other concrete applications include for instance the control of supply chains, gas pipelines and traffic flows, as discussed in \cite{BCGHP14} and \cite{YK22}.
Most of physical models are \imp{nonlinear systems} and there are already works on the control of 1D first-order nonlinear hyperbolic systems.
However, in this survey, we only focus on the control problems of \imp{linear} (and time-independent) systems.
One of the main reasons is that some of the fundamental control problems are still far from being well-understood, even in this simplified context.
We shall mention few references for nonlinear systems at the end of the paper.

\paragraph{General framework.}

The system we are interested in is the following system of 1D hyperbolic Partial Differential Equations (PDEs):
\begin{equation}\label{syst}
\begin{dcases}
\pt{y}(t,x)+\Lambda(x) \px{y}(t,x)=M(x) y(t,x), \\
y_-(t,1)=u(t), \quad y_+(t,0)=Qy_-(t,0), \\
y(0,x)=y^0(x),
\end{dcases}
\quad (t,x) \in R_T.
\end{equation}
The notations are as follows.
The domain is the rectangle $R_T=(0,T) \times (0,1)$, where $T>0$.
$t$ is the time variable and $x$ is the space variable.
Since we are interested in \imp{systems}, the state $y:R_T \to \R^n$ is thus a vector-valued function ($n \geq 2$).

The $n \times n$ matrix $\Lambda$ will always be assumed diagonal $\Lambda =\diag(\lambda_1,\ldots,\lambda_n)$, with $m \geq 1$ negative and $p \geq 1$ positive entries ($m+p=n$), also called \imp{speeds}:
$$
\lambda_1(x), \ldots, \lambda_m(x)<0<\lambda_{m+1}(x),\ldots,\lambda_{m+p}(x),
$$
for every $x$.
We shall sometimes use the block decomposition $\Lambda=\diag(\Lambda_-,\Lambda_+)$, where
\begin{equation}\label{decomp Lambda}
\Lambda_-=\diag(\lambda_1,\ldots,\lambda_m), \quad \Lambda_+=\diag(\lambda_{m+1},\ldots,\lambda_{m+p}).
\end{equation}
The speeds will be assumed to be Lipschitz continuous on $\R$, so that the \imp{characteristics} are globally defined: for every $1 \leq k \leq n$ and $(t,x) \in \R^2$, there is a unique $C^1$ function $s \mapsto \chi_k(s;t,x)$ solution to the Ordinary Differential Equation (ODE)
$$
\begin{dcases}
\dds \chi_k(s;t,x)=\lambda_k(\chi_k(s;t,x)), \quad \forall s \in \R, \\
\chi_k(t;t,x)=x.
\end{dcases}
$$
We point out that two speeds may have common values, or even cross, and the system is therefore not necessarily strictly hyperbolic.
Additional assumptions may be added in some sections below though.

The $n \times n$ matrix $M$ couples the equations of the system inside the domain and will be called the \imp{internal coupling matrix}.
We will assume that $M \in L^{\infty}(0,1)^{n \times n}$.

The structure of $\Lambda$ induces a natural splitting of the state $y$ into components corresponding to negative and positive speeds:
$$
y=
\begin{pmatrix}
y_- \\
y_+ 
\end{pmatrix},
\quad
\begin{array}{l}
y_- \in \R^m,
\\
y_+ \in \R^p.
\end{array}
$$
This explains why, for system \eqref{syst} to be well-posed, we need to consider boundary conditions at $x=1$ for $y_-$ and at $x=0$ for $y_+$.
The particular type of (local) boundary conditions considered in \eqref{syst} is also motivated by applications (\cite{BC16,YK22}).
The function $u:(0,T) \to \R^m$ is called the \imp{control}, it is a function that will be at our disposal.
Since it acts on the boundary of the domain, we speak of \imp{boundary controllability}.
We will also speak of ``the controls'' when referring to its components $u_1,\ldots,u_m$.
We have $m$ controls for $n$ equations so that we have less controls than equation.
Such systems are called \imp{underactuated}.
Since the control acts only on one side of the boundary, we speak of \imp{one-sided controllability} (\cite{Li10}).
The system is coupled on the other part of the boundary by the matrix $Q \in \R^{p \times m}$.
This matrix will be called the \imp{boundary coupling matrix}.

We shall also refer to systems of the form \eqref{syst} as \imp{$m+p$ systems} and denote them by $(\Lambda,Q,M)$.
Systems with $M=0$ are called (linear) \imp{systems of conservation laws} (\cite{Bre00,BC16}), whereas systems with $M \neq 0$ are sometimes referred to as (linear) \imp{systems of balance laws} (\cite{BCGHP14,BC16}).

System \eqref{syst} is well-posed in several functional settings.
In this paper, we will work in $L^2$ for the state and the controls.
We can prove that, for every $y^0 \in L^2(0,1)^n$ and $u \in L^2(0,T)^m$, there exists a unique solution
$$
y \in C^0([0,T];L^2(0,1)^n) \cap C^0([0,1];L^2(0,T)^n).
$$
By solution we mean ``solution along the characteristics'' (sometimes also called ``broad solution'', see \cite[Chapter 3]{Bre00}).
We refer for instance to \cite[Appendix A]{CHOS21} for a proof of the above well-posedness result in such a setting (see also \cite[Appendix A]{BC16} when $u=0$ and \cite[Lemma 3.2]{CN19} in the $L^\infty$ setting).

\paragraph{Control problems.}

Since $y(t,\cdot) \in L^2(0,1)^n$ for all $t$, we can consider control problems in the state space $H=L^2(0,1)^n$:

\begin{definition}
Let $T>0$ be fixed.
We say that system \eqref{syst} is:
\begin{itemize}
\item
\imp{exactly controllable} (below, \EC) in time $T$ if, for every $y^1, y^0 \in L^2(0,1)^n$, there exists $u \in L^2(0,T)^m$ such that the corresponding solution $y$ to system \eqref{syst} satisfies the final condition $y(T,\cdot)=y^1$.

\item
\imp{null controllable} (below, \NC) in time $T$ if
the previous property holds at least for $y^1=0$.
\end{itemize}
\end{definition}

Obviously, \EC always implies \NC.
In general, the converse implication does not hold.
Nevertheless, for our system \eqref{syst}, we have equivalence if $\rank Q=p$ (see e.g. \cite[Remark 1.3]{HO21-JMPA}).
We can also prove that, if $\rank Q<p$, then the system is never \EC (see Corollary \ref{cor main thm EC} below).

It is not difficult to check that controllability in time $T_1$ implies controllability in any time $T_2 \geq T_1$.
It is therefore natural and fundamental to try to find the smallest possible control time, the so-called \imp{minimal control time}.

\begin{definition}
We denote by $\TinfEC(\Lambda,Q,M) \in [0,+\infty]$ the minimal time for the \EC of system \eqref{syst}:
$$
\TinfEC(\Lambda,Q,M)=
\inf\ens{T>0 \st \text{System \eqref{syst} is \EC in time $T$}}.
$$
Similarly, we denote by $\TinfNC(\Lambda,Q,M) \in [0,+\infty]$ the minimal time for the \NC of system \eqref{syst}.
\end{definition}

We say ``minimal'' control time according to the current literature, despite this time is not always a minimal element of the set (specifically for our system \eqref{syst}, this follows from the results of \cite{CN19}, or a combination of the results of \cite{HO21-JMPA} with \cite{Rus78}).
We keep this naming here, but we use the notation with the ``inf'' to avoid eventual confusions.

It is not difficult to see that, whether we consider \EC or \NC, the minimal control time for a single equation (the transport equation) with speed $\lambda_k$ is the solution $T_k$ to $\chi_k(T_k;0,1)=0$ if $\lambda_k<0$ and to $\chi_k(T_k;0,0)=1$ if $\lambda_k>0$.
In any case, it is explicitly given by
\begin{equation}\label{def Tk}
T_k=\int_0^1 \frac{1}{\abs{\lambda_k(\xi)}} \, d\xi.
\end{equation}
We refer for instance to \cite[Section 3]{Rus78}, \cite[Section 3.1]{CVKB13} or \cite[Chapter 1]{Oli25-notes}.
These quantities will play a key role below.
The related problem is however much more challenging for the strongly coupled system \eqref{syst}.

\paragraph{Motivation of the survey.}

The controllability of 1D first-order hyperbolic systems with boundary controls has been widely studied in the literature.
It dates back at least to the article \cite{Rus67} and the celebrated survey \cite{Rus78} (see also \cite{Cir69}).
After these seminal works, it seems that the community turned its attention to quasilinear systems (i.e. $\Lambda$ depending on $y$) in the $C^1$ framework of so-called semi-global solutions (see e.g. \cite{LR03,Li10,Hu15}).
However, lately there has been a resurgence on finding the minimal control time for linear systems.
This was initiated in \cite{CN19} and followed by a series of works \cite{HO21-JMPA,CN21,HO21-COCV,CN21-pre,HO22-JDE,HO25-SICON,HO25-JDE}.
The goal of this survey is to present some of these important recent advances and to give an overview of the type of control techniques that have been developed to address this problem.

\section{The result of Russell and the method of Li and Rao}
\label{sect LR}

We start this survey by presenting the first historical result on the controllability of system \eqref{syst}, obtained in \cite{Rus78}, and the constructive method of Li and Rao (\cite{LR03}), used as the modern approach to prove this result (\cite{Li10}).

All along this part, we will assume in addition that the speeds \imp{do not cross}.
Since they are continuous functions, we can then always sort them as follows:
\begin{equation}\label{hyp speeds are ordered}
\lambda_1 \leq \ldots \leq \lambda_m<0<\lambda_{m+1} \leq \ldots \leq \lambda_{m+p}.
\end{equation}
Observe that this implies the following order (see \eqref{def Tk}):
\begin{equation}\label{order times}
T_1 \leq \cdots \leq T_m \quad \text{ and } \quad T_{m+1} \geq \cdots \geq T_{m+p}.
\end{equation}

The first result on the controllability of system \eqref{syst} is the following:

\begin{theorem}[\cite{Rus78,Li10}]\label{thm LR EC}
Assume \eqref{hyp speeds are ordered} and $\rank Q=p$.
Then, system \eqref{syst} is \EC in time $T=T_m+T_{m+1}$.
\end{theorem}

\begin{proof}
We give the proof in \cite[Section 3.4]{Li10}, which is based on the constructive method of Li and Rao (\cite{LR03}).
The basic idea is to construct a ``solution'' $y$ to
\begin{equation}\label{proof LR:cond for y}
\begin{dcases}
\pt{y}(t,x)+\Lambda(x) \px{y}(t,x)=M(x) y(t,x), \\
y_+(t,0)=Qy_-(t,0), \\
y(0,x)=y^0(x), \quad y(T,x)=0,
\end{dcases}
\quad (t,x) \in R_T.
\end{equation}
Once such a $y$ has been constructed, the control is simply the trace $u(t)=y_-(t,1)$.
The construction of $y$ is done in several steps as follows.
\begin{enumerate}
\item
\textbf{Forward problem.}
We first consider the ``forward in time'' system
$$
\begin{dcases}
\pt{y^f}(t,x)+\Lambda(x) \px{y^f}(t,x) = M(x) y^f(t,x), \\
y^f_-(t,1) = u^f(t), \quad y^f_+(t,0)=Qy^f_-(t,0), \\
y^f(0,x) = y^0(x),
\end{dcases}
\quad t \in (0,T_m),
$$
where $u^f$ is arbitrary.
Then, we set
$$a(t)=y^f(t,0), \quad t \in (0,T_m).$$

\item
\textbf{Backward problem.}
By assumption on $Q$, there exists $S \in \R^{m \times p}$ such that $QS=\Id_p$.
We then consider the ``backward in time'' system
$$
\begin{dcases}
\pt{y^b}(t,x)+\Lambda(x) \px{y^b}(t,x)=M(x) y^b(t,x), \\
y^b_-(t,0)=S y^b_+(t,0), \quad y^b_+(t,1)=u^b(t), \\
y^b(T,x)=y^1(x),
\end{dcases}
\quad t \in (T-T_{m+1},T),
$$
where $u^b$ is arbitrary.
This system is well-posed (for instance, we can use the changes of variables $t \mapsto T-t$ and $x \mapsto 1-x$).
Then, we set
$$b(t)=y^b(t,0), \quad t \in (T-T_{m+1},T).$$

\item
\textbf{Gluing the boundary conditions.}
By assumption on $T$, we have $T_m \leq T-T_{m+1}$.
Therefore, we can define
$$
c(t)=
\begin{dcases}
a(t) & \text{ if } t \in (0,T_m), \\
b(t) & \text{ if } t \in (T-T_{m+1}, T).
\end{dcases}
$$
By construction, it satisfies the boundary condition at $x=0$ for all $t$:
\begin{equation}\label{BC for c}
c_+(t)=Qc_-(t), \quad t \in (0,T).
\end{equation}

\item
\textbf{Changing $t$ and $x$.}
We now change to role of $t$ and $x$ and consider the function $c$ defined in the previous step as an initial condition.
We thus consider the system
$$
\begin{dcases}
\px{y}(t,x)+\Lambda(x)^{-1} \pt{y}(t,x)=\Lambda(x)^{-1} M(x) y(t,x), \\
y_-(T,x)=y^1_-(x), \quad y_+(0,x)=y^0_+(x), \\
y(t,0)=c(t),
\end{dcases}
\quad (t,x) \in R_T.
$$
This ``time'' dependent system is well-posed.
Its solution $y$ is our candidate to be also the solution to \eqref{proof LR:cond for y}.

\item
\textbf{Uniqueness in triangular-like domains.}
The function $y$ constructed in the previous step obviously satisfies the equations in \eqref{proof LR:cond for y} and the boundary condition at $x=0$ (see \eqref{BC for c}).
However, only parts of the initial and final conditions are satisfied.
Nevertheless, using assumption \eqref{hyp speeds are ordered} we can show that $y$ necessarily coincides with $y^f$ in some triangular-like domain:
$$
y=y^f \quad \text{ in } \quad \ens{(t,x) \st t<T_m, \quad 0<x<\chi_m(t;0,1)}.
$$
Since $(0,x)$ belongs to this domain, the initial condition is therefore satisfied: $y(0,x)=y^f(0,x)=y^0(x)$.
The final condition is obtained similarly by noting that $y=y^b$ in $\ens{(t,x) \st t>T-T_{m+1}, \quad 0<x<\chi_{m+1}(t;T,1)}$, which contains $(T,x)$.

\end{enumerate}

We refer to \cite[Section 3.4]{Li10} or \cite[Chapter 4]{Oli25-notes} for more details.
\end{proof}

\begin{remark}\label{rem other app of LR}
The method of Li and Rao can be adapted to deal with many situations (e.g. two-sided controllability, two-sided controllability with less controls, \NC, etc.) and, more importantly, it works for quasilinear systems too.
We refer to the reference book \cite{Li10}.
\end{remark}

\begin{remark}\label{rem Li sharp}
It is written in \cite{Li10} that the control time $T_m+T_{m+1}$ is ``sharp''.
However, let us warn the reader that this does not mean that $T_m+T_{m+1}$ is the minimal control time.
What is meant in \cite{Li10} is that, ``in general'', we cannot do better than $T_m+T_{m+1}$, where ``in general'' means ``for some $Q$ (with $\rank Q=p$) and some $M$''.
However, we will see below that, for most matrices $Q$ and $M$, the control time $T_m+T_{m+1}$ is not the minimal one.
\end{remark}

\section{The minimal control time for systems of conservation laws}
\label{sect cons laws}

In the survey \cite{Rus78}, the author also already observed that the control time $T_m+T_{m+1}$ can be larger than the minimal control time, even in the simple case of \imp{systems of conservation laws}, that is systems with $M=0$:

\begin{equation}\label{syst cons laws}
\begin{dcases}
\pt{y}(t,x)+\Lambda(x) \px{y}(t,x)=0, \\
y_-(t,1)=u(t), \quad y_+(t,0)=Q y_-(t,0), \\
y(0,x)=y^0(x),
\end{dcases}
\quad (t,x) \in R_T.
\end{equation}
We point out that $u=0$ is a natural control that brings the solution of system \eqref{syst cons laws} to $0$ in time $T_m+T_{m+1}$.
In most situations though, a better choice of control possibly brings the solution to $0$ faster, in a smaller time.
In \cite{Rus78} the author started to look more closely at properties of the matrix $Q$ that could lead to a characterization of the minimal control time.
However, he did not succeed and he left this as an open problem (\cite[Remark p. 656]{Rus78}).
This problem was completely solved few years later in \cite{Wec82}.
In that paper, the author obtained an explicit expression of the minimal control time for \NC\footnote{Let us warn the reader that, in \cite{Wec82}, the author calls ``exact controllability'' what is in fact commonly called nowadays null controllability.} of the system of conservation laws \eqref{syst cons laws} in terms of some indices related to $Q$.
In particular, he showed in \cite[Section 4]{Wec82} that none of the attempts in \cite{Rus78} produced the minimal control time.

Finding the minimal control time is a problem that was also independently solved in \cite{CN19} but assuming in addition $Q \in \setCN_p$ (see \eqref{def setCN} below for the definition of this set), and also in \cite{HO21-JMPA} but assuming $\rank Q=p$.
Finally, the problem was then revisited in its generality in \cite{HO22-JDE}.
It is this last approach that we wish to present in Section \ref{sect cons laws}.
There are two notable advantages over the work \cite{Wec82}: 1) the formula obtained for the minimal control time with this approach is simpler, 2) the proof does not use the so-called duality, which immediately allows the approach to be extended to other functional frameworks, such as $C^1$ instead of $L^2$.

\subsection{Canonical decomposition of boundary coupling matrices}

We start with a key definition.

\begin{definition}\label{def canon}
A boundary coupling matrix $Q \in \R^{p \times m}$ is in canonical form if it has at most one nonzero entry on each row and each column, and this entry is equal to $1$.
We denote by $(r_1,c_1), \ldots, (r_{\rho},c_{\rho})$ the positions of the corresponding nonzero entries, with $r_1<\cdots<r_{\rho}$.
\end{definition}

\begin{example}\label{ex Qzero}
The following matrix is in canonical form:
$$
\begin{pmatrix}
0 & \fbox{$1$} & 0 \\
0 & 0 & \fbox{$1$} \\
0 & 0 & 0 \\
\fbox{$1$} & 0 & 0
\end{pmatrix}
\quad
\begin{array}{l}
\rightarrow (r_1,c_1)=(1,2), \\
\rightarrow (r_2,c_2)=(2,3), \\
\rightarrow \times \, \text{not a $r_k$,} \\
\rightarrow (r_3,c_3)=(4,1).
\end{array}
$$
\end{example}

\begin{remark}\label{rem Tinf can syst}
Boundary coupling matrices in canonical form yield the simplest possible systems, and the minimal control time of such systems is trivially determined.
Indeed, because of this very simple coupling structure, the system consists in uncoupled transport equations and $1+1$ systems.
It is then clear that the system is:
\begin{itemize}
\item
\NC in time $T$ if, and only if,
\begin{equation}\label{cond T NC}
T \geq \max \ens{T_m, T_{m+1}, T_Q},
\quad T_Q=\max_{k \leq \rho} \, T_{m+r_k}+T_{c_k},
\end{equation}
(moreover, $u=0$ is a control).
Above, we omit $T_Q$ in case $Q=0$.

\item
\EC in time $T$ if, and only if, $Q$ has no zero rows and (since $r_k=k$ for all $k$ in that case)
\begin{equation}\label{cond T EC}
T \geq \max \ens{T_m, T_Q},
\quad T_Q=\max_{k \leq p} \, T_{m+k}+T_{c_k}.
\end{equation}
\end{itemize}

\end{remark}

Any matrix can be put in a unique canonical form:

\begin{theorem}\label{thm can decomp BC mat}
For every $Q \in \R^{p \times m}$, there exists a unique $\Qcan \in \R^{p \times m}$ which is in canonical form and such that
$$LQU=\Qcan,$$
for some unit lower triangular matrix $L \in \R^{p \times p}$ and some invertible upper triangular matrix $U \in \R^{m \times m}$.
We call $\Qcan$ the canonical form of $Q$.
\end{theorem}

With the above result, we can extend the definition of the indices $(r_1,c_1), \ldots, (r_{\rho},c_{\rho})$ to any nonzero matrix:

\begin{definition}\label{def rkck can form}
For any nonzero matrix $Q \in \R^{p \times m}$, we denote by $(r_1,c_1), \ldots, (r_{\rho},c_{\rho})$ the positions of the nonzero entries of its canonical form ($r_1<\cdots<r_{\rho}$).
\end{definition}

The existence part in the theorem is just a consequence of the Gaussian elimination.
The matrix $L$ corresponds to downward row substitutions and the matrix $U$ corresponds to rightward column substitutions and then normalization to $1$ of the remaining nonzero entries.

\begin{example}\label{ex Q}
The canonical form of the boundary coupling matrix
$$
Q=
\begin{pmatrix}
0 & 1 & 2 \\
0 & 2 & 5 \\
0 & 1 & 2 \\
4 & -4 & 4
\end{pmatrix},
$$
is the matrix of Example \ref{ex Qzero} (see e.g. \cite[Example 2.6.6]{Oli25-notes} for details).
\end{example}

The uniqueness of the decomposition is less straightforward though, because there may be several pairs $(L,U)$ that produce the canonical form of $Q$.
For a proof, we refer for instance to \cite[Section 2.6]{Oli25-notes}.

\subsection{The minimal control time}

We are now in position to completely and simply characterize the controllability properties of the system of conservation laws \eqref{syst cons laws}:

\begin{theorem}[\cite{HO22-JDE}]\label{main thm syst cons laws}
Assume \eqref{order times}.
Let $(r_k,c_k)$ be the indices of $Q$ (Definition \ref{def rkck can form}).
Let $T>0$.
System \eqref{syst cons laws} is:
\begin{itemize}
\item
\NC in time $T$ if, and only if, $T$ satisfies \eqref{cond T NC} (as before, we omit $T_Q$ in the formula in case $Q=0$).

\item
\EC in time $T$ if, and only if, $\rank Q=p$ and $T$ satisfies \eqref{cond T EC}.
\end{itemize}
\end{theorem}

The proof, taken from \cite{HO22-JDE}, is based on the notion of \imp{equivalent system} (other proofs are possible, see e.g. \cite[Section 2.5]{Oli25-notes}, and recall that the minimal time for the \NC was first found in \cite{Wec82}, however with a different formula).
The main idea is to transform system \eqref{syst cons laws} into a system that is equivalent from a controllability point of view and that has a simpler coupling structure (i.e. a simpler $Q$).
Let us make this notion precise here.
In what follows, it will be convenient to refer to a system of the form \eqref{syst cons laws} as
$$(\Lambda,Q).$$
We say that $y$ is a solution to $(\Lambda,Q)$ in $(0,T)$ if it is a solution to system \eqref{syst cons laws} for some initial data $y^0$ and function $u$.

\begin{definition}\label{def syst equiv}
We say that two systems $(\Lambda,Q_1)$ and $(\Lambda,Q_2)$ are equivalent, and we write
$$(\Lambda,Q_1) \sim (\Lambda,Q_2),$$
if there exists an isomorphism $L:L^2(0,1)^n \longrightarrow L^2(0,1)^n$ such that, for every $T>0$, $y$ is a solution to $(\Lambda,Q_1)$ in $(0,T)$ if, and only if, $\tilde{y}(t,\cdot)=L(y(t,\cdot))$ is a solution to $(\Lambda,Q_2)$ in $(0,T)$.
\end{definition}

It is easy to check that $\sim$ is an equivalence relation.
It is also not difficult to check that two equivalent systems share the same controllability properties.

\begin{remark}
The notion of equivalent systems introduced in the previous definition is not concerned with how the control from one system is obtained from the control of the other system.
It is different from the notion of \imp{feedback equivalence} introduced in the seminal work \cite{Bru70} in finite dimension (see also \cite[Section 4]{Cor15}), which was designed to transfer the \imp{stabilization} properties of one system to another and thus required a more specific link between the two systems.
\end{remark}

The key result is the following.

\begin{proposition}\label{thm equiv syst}
Assume \eqref{order times}.
Then, $(\Lambda,Q) \sim (\Lambda,LQU)$ for any invertible lower triangular matrix $L \in \R^{p \times p}$ and any invertible upper triangular matrix $U \in \R^{m \times m}$.
\end{proposition}

Applying this result, we have in particular that
$$(\Lambda,Q) \sim (\Lambda,\Qcan),$$
where $\Qcan$ is the canonical form of $Q$ (see Theorem \ref{thm can decomp BC mat}).
Therefore, Theorem \ref{main thm syst cons laws} is an immediate consequence of Proposition \ref{thm equiv syst} and Remark \ref{rem Tinf can syst}.

\begin{proof}[Proof of Proposition \ref{thm equiv syst}]
This is a special case of \cite[Proposition 4.1]{HO22-JDE}.
For simplicity, we give the proof in the case of constant speeds $\lambda_k(x)=\lambda_k$ (then assumption \eqref{order times} is equivalent to \ref{hyp speeds are ordered}).

\begin{itemize}
\item
We first prove that
$$(\Lambda,Q) \sim (\Lambda,LQ).$$
To this end, we are going to use the spatial transformation
$$
\tilde{y}_-=y_-, \quad
\tilde{y}_{m+\alpha}(t,x)=
\sum_{\beta \leq \alpha} l_{\alpha\beta} y_{m+\beta}\left(t, \frac{\lambda_{m+\beta}}{\lambda_{m+\alpha}} x\right) \quad (1 \leq \alpha \leq p).
$$
Note that $(\lambda_{m+\beta}/ \lambda_{m+\alpha}) x \in (0,1)$ for $\beta \leq \alpha$ thanks to \eqref{hyp speeds are ordered}.
We easily check that this transformation is invertible, with inverse given by
\begin{equation}\label{proof equiv invers transfo}
y_-=\tilde{y}_-, \quad
y_{m+\beta}(t,x)=
\sum_{\gamma \leq \beta} l^{\beta\gamma} \tilde{y}_{m+\gamma}\left(t, \frac{\lambda_{m+\gamma}}{\lambda_{m+\beta}} x\right) \quad (1 \leq \beta \leq p),
\end{equation}
where $l^{\beta\gamma}$ denotes the $(\beta,\gamma)$ entry of $L^{-1}$.

\item
Assume now that $y$ is a solution to the system $(\Lambda,Q)$ for some $(y^0,u)$ and let us show that $\tilde{y}$ is then a solution to the system $(\Lambda,LQ)$ for some $(\tilde{y}^0,\tilde{u})$.

\begin{itemize}
\item
\textbf{Initial condition.}
Obviously, it is given by
$$
\tilde{y}_-^0=y_-^0, \quad
\tilde{y}^0_{m+\alpha}(x)=
\sum_{\beta \leq \alpha} l_{\alpha\beta} y^0_{m+\beta}\left(\frac{\lambda_{m+\beta}}{\lambda_{m+\alpha}} x\right) \quad (1 \leq \alpha \leq p).
$$

\item
\textbf{Boundary conditions.}
At $x=0$ it is clearly satisfied since $\tilde{y}_-=y_-$ and $\tilde{y}_+(t,0)=Ly_+(t,0)$.
At $x=1$ this just defines the control $\tilde{u}=u$.

\item
\textbf{Equations.}
It is clear that $\tilde{y}_k=y_k$ satisfies the desired equation for $k \leq m$.
Let us now fix $1 \leq \alpha \leq p$.
Using the equations satisfied by $y_+$, we have
\begin{multline*}
\pt{\tilde{y}_{m+\alpha}}(t,x)
+\lambda_{m+\alpha}(x)\px{\tilde{y}_{m+\alpha}}(t,x)
=
\\
\sum_{\beta \leq \alpha} l_{\alpha\beta}
\underbrace{\left(
-\lambda_{m+\beta}
+\lambda_{m+\alpha} \frac{\lambda_{m+\beta}}{\lambda_{m+\alpha}}
\right)}_{=0}
\px{y_{m+\beta}} \left(t, \frac{\lambda_{m+\beta}}{\lambda_{m+\alpha}} x\right).
\end{multline*}
Consequently, $\tilde{y}_{m+\alpha}$ also satisfies the desired equation.

\end{itemize}

Similar computations show that, conversely, if $\tilde{y}$ is a solution to the system $(\Lambda,LQ)$ for some $(\tilde{y}^0,\tilde{u})$, then $y$ defined by the inverse transformation \eqref{proof equiv invers transfo} is a solution to the system $(\Lambda,Q)$ for some $(y^0,u)$.

\item
Similarly, we can show that
$$(\Lambda,Q) \sim (\Lambda,QU),$$
by using the spatial transformation
$$
\tilde{y}_\alpha(t,x)=
\sum_{\beta \geq \alpha} u^{\alpha\beta} y_\beta\left(t, \frac{\lambda_{\beta}}{\lambda_{\alpha}} x\right) \quad (1 \leq \alpha \leq m),
\quad \tilde{y}_+=y_+,
$$
where $u^{\alpha\beta}$ denotes the $(\alpha,\beta)$ entry of $U^{-1}$.
Let us just mention that the control is then
\begin{equation}\label{link ut u}
\tilde{u}_\alpha(t)
=\sum_{\beta \geq \alpha} u^{\alpha\beta} y_\beta \left(t, \frac{\lambda_{\beta}}{\lambda_{\alpha}}\right) \quad (1 \leq \alpha \leq m),
\end{equation}
which is well defined since $y \in C^0([0,1];L^2(0,T)^n)$.

\end{itemize}
This concludes the proof of Proposition \ref{thm equiv syst}.
\end{proof}

\begin{remark}\label{M=0 feedback}
Since $\tilde{u}=0$ is a control for the \NC of $(\Lambda,\Qcan)$ (see Remark \ref{rem Tinf can syst}), plugging this into formula \eqref{link ut u} (recall that $(\Lambda,Q)$ and $(\Lambda,LQ)$ have the same controls) and extracting the term for $\beta=\alpha$ from the sum to obtain the control $u_\alpha$, we deduce that (still assuming that the speeds are constant)
\begin{equation}\label{syst hom:feedback law}
u_\alpha(t)=- \frac{1}{u^{\alpha\alpha}} \sum_{\beta>\alpha} u^{\alpha\beta} y_\beta\left(t, \frac{\lambda_{\beta}}{\lambda_{\alpha}}\right) \quad (1 \leq \alpha \leq m).
\end{equation}
This gives a formula of the control only in function of the state $y$ at time $t$: $u(t)=F(y(t,\cdot))$.
The linear operator $F$ is called the \imp{feedback law}.
Note that it is an unbounded operator in our case (it is not defined on the state space $L^2(0,1)^n$).
Formula \eqref{syst hom:feedback law} can also be used to write the control $u$ explicitly in function of the initial condition (in fact, $y^0_-$).
Let us mention that a similar feedback law was first obtained in \cite[Section 7]{CN19}, but under the assumption that $Q \in \setCN_{\rhoCN}$ for $\rhoCN=\min\ens{m-1,p}$ (see \eqref{def setCN} below for the definition of this set).
\end{remark}

\section{The minimal time for exact controllability}
\label{sect EC}

We now address the problem of finding the minimal control time for arbitrary $M$.
It was first suggested in \cite{Rus78} that the minimal control time for the \EC could be independent of $M$.
This was investigated by using some perturbation arguments (\cite[Theorem 3.7 and p.~659]{Rus78}).
However, in these results, either $M$ is assumed to be small (in which case the result is in fact not surprising since \EC is a property that is stable by small bounded perturbations, see e.g. \cite[Theorem 4.1]{DR77}), or $M$ is assumed to be such that the corresponding system is approximately controllable (which is in general not easier to check).
This problem was also investigated more recently in \cite{CN19}.
A complete solution was finally obtained in \cite{HO21-JMPA}, under the following assumption:
\begin{equation}\label{hyp NRL86}
\forall i \neq j, \quad (\lambda_i(x_0)=\lambda_j(x_0) \text{ for some } x_0 \in [0,1])
\quad \Longrightarrow \quad
m_{ij}=0.
\end{equation}
The goal of this part is to present the result obtained therein:

\begin{theorem}[\cite{HO21-JMPA}]\label{main thm HO21-JMPA}
Assume \eqref{hyp NRL86}.
Then, the minimal control time for the \EC of system \eqref{syst} does not depend on $M$:
$$\TinfEC(\Lambda,Q,M) =\TinfEC(\Lambda,Q,0).$$
\end{theorem}

We recall that $\TinfEC(\Lambda,Q,0)$ is explicitly known from the previous part about systems of conservation laws (Theorem \ref{main thm syst cons laws}).
Consequently, we immediately deduce:

\begin{corollary}\label{cor main thm EC}
Assume \eqref{hyp NRL86}.
We have:
\begin{enumerate}
\item
$\TinfEC(\Lambda,Q,M)<+\infty$ if, and only if, $\rank Q=p$.

\item\label{cor main thm EC i2}
Assume that $\rank Q=p$ and \eqref{order times} hold.
Let $c_k$ be the indices of $Q$ (Definition \ref{def rkck can form}).
Then,
$$\TinfEC(\Lambda,Q,M)=\max \ens{T_m, T_Q},
\quad T_Q=\max_{k \leq p} \, T_{m+k}+T_{c_k}.$$
\end{enumerate}
\end{corollary}

\begin{remark}
Even if the second point of the corollary is clearly the most important result, the implication ``$\Longrightarrow$'' in the first point is already non trivial when $M \neq 0$.
Notably, this answers by the negative the open problems 3, 4 and 5 in \cite{LR10-SCAM}, in the linear setting.
\end{remark}

\begin{remark}
Assumption \eqref{hyp NRL86} has been introduced in \cite{NRL86}.
If it is not satisfied, then, in general, the conclusion of Theorem \ref{main thm HO21-JMPA} is no longer true (see \cite[Appendix C]{HO21-JMPA}).
It would be interesting to find the optimal assumptions on $\Lambda,Q,M$ such that the conclusion of the theorem remains valid.
Let us also mention that, in the survey \cite{Rus78}, the author assumed the following different condition:
$$
\forall i,j, \quad (\lambda_i(x_0)=\lambda_j(x_0) \text{ for some } x_0 \in [0,1])
\quad \Longrightarrow \quad
\lambda_i=\lambda_j.
$$
This condition is essentially stronger than \eqref{hyp NRL86} (see \cite[Remark 4.6]{HO21-JMPA}).
It also implies \eqref{hyp speeds are ordered}, after eventually switching some equations of the system.
\end{remark}

The next sections are devoted to the proof of Theorem \ref{main thm HO21-JMPA}.
It is a development of the original ideas of \cite{Rus78}, combined with some results of \cite{DO18} and \cite{NRL86}.

\subsection{Abstract setting}\label{sect abst sett}

To understand the main ideas of the proof of Theorem \ref{main thm HO21-JMPA}, we need to consider system \eqref{syst} from an abstract point of view.
As many other linear PDEs, it is well-known that system \eqref{syst} can be seen as an abstract ODE (see e.g. \cite{Cor07,TW09}):
\begin{equation}\label{abst syst}
\begin{dcases}
\ddt y(t)= \opA{M}y(t)+B u(t), \quad t \in (0,T), \\
y(0)=y^0.
\end{dcases}
\end{equation}
The state and control spaces are
$$H=L^2(0,1;\C)^n, \quad U=\C^m.$$
The unbounded linear operator $\opA{M}: D(\opA{M}^*) \subset H \longrightarrow H$ is 
\begin{align*}
(\opA{M}y)(x) &=-\Lambda(x) y'(x)+M(x) y(x),
\\
D(\opA{M}) &=\ens{y \in H^1(0,1)^n \st \quad y_-(1)=0, \quad y_+(0)=Qy_-(0)}.
\end{align*}
We can easily check that $D(\opA{M})$ is dense in $H$ and that
$$
D(\opA{M}^*) =\ens{z \in H^1(0,1)^n \st \quad z_-(0)=\Qadj^\tr z_+(0), \quad z_+(1)=0},
$$
where
\begin{equation}\label{def Qadj}
\Qadj=-\Lambda_+(0)Q\Lambda_-(0)^{-1},
\end{equation}
(recall \eqref{decomp Lambda}).
Since $D(\opA{M}^*)$ does not depend on $M$, we shall simply denote it by $D(\opA{}^*)$.
Finally, the control operator $B \in \lin{U,D(\opA{}^*)'}$ is
$$
\ps{Bu}{z}{D(\opA{}^*)',D(\opA{}^*)}=- \ps{u}{\Lambda_{-}(1) z_-(1)}{U}.
$$

Even though all the parameters $\Lambda,Q$ and $M$ are real-valued, we have to consider complex spaces in this part because we will use below a condition involving the spectral elements of the operator $\opA{M}$.
This does not change anything regarding the controllability of the system.

It can be shown that the operator $\opA{M}$ generates a $C_0$-semigroup $(\opS{M}(t))_{t \geq 0}$ on $H$ (see e.g. \cite[Appendix A]{BC16}) and that $H$ is admissible for $(\opA{M},B)$ (see e.g. \cite[p. 401]{DO18}), which means that we have the following inequality for all $T>0$:
\begin{equation}\label{admi ineq}
\int_0^T \norm{B^*\opS{M}(t)^*z}_U^2 \, dt \leq C \norm{z}_H^2,
\end{equation}
for every $z \in D(\opA{}^*)$ and some $C>0$ independent of $z$.

In what follows, system \eqref{abst syst} will also be referred to as $(\opA{M},B)$ and the minimal control time will be denoted by $\TinfEC(\opA{M},B)$.

\subsection{The compactness-uniqueness method}
\label{sect compact uniq}

It is well-known that controllability has a dual concept named \imp{observability}.
An application of a standard result of functional analysis gives the following characterization:

\begin{theorem}[\cite{DR77}]\label{thm duality}
System $(\opA{M},B)$ is \EC in time $T$ if, and only if, there exists $C>0$ such that
\begin{equation}\label{obs ineq}
\norm{z}_H^2 \leq C\int_0^T \norm{B^*\opS{M}(t)^*z}_U^2 \, dt,
\end{equation}
for every $z \in D(\opA{}^*)$.
\end{theorem}

Inequality \eqref{obs ineq} is called an \imp{observability inequality}.
This is the ``reversed'' inequality of \eqref{admi ineq}.
One powerful method to establish such inequalities is the \imp{compactness-uniqueness method} (see e.g. \cite{RT74,DR77,Rus78,Zua87,CT10}):

\begin{theorem}[\cite{DO18}]\label{thm DO18}
Assume that:
\begin{enumerate}
\item
$(\opA{M},B)$ satisfies the Fattorini-Hautus test:
\begin{equation}\label{FH test}
\ker(\lambda-\opA{M}^*) \cap \ker B^*=\ens{0}, \quad \forall \lambda \in \C.
\end{equation}

\item
There exist $T_0>0$, a compact operator $K:H \rightarrow E$ (for some Banach space $E$) and $C>0$ such that
\begin{equation}\label{obs with cpct rem}
\norm{z}_H^2 \leq C\left(\int_0^{T_0} \norm{B^*\opS{M}(t)^*z}_U^2 \, dt+\norm{K z}_E^2\right),
\end{equation}
for every $z \in D(\opA{}^*)$.

\end{enumerate}
Then, $(\opA{M},B)$ is \EC in any time $T>T_0$.
\end{theorem}

The Fattorini-Hautus test \eqref{FH test} is always a necessary condition for the \EC to hold in some time.
It is easy to check that is true for our system $(\opA{M},B)$, for any $M$.

Estimate \eqref{obs with cpct rem} is called an \imp{observability inequality with compact remainder}.
Now, given two systems $(\opA{M},B)$ and $(\opA{M_0},B)$, a standard procedure to try to prove such estimates is to start with the observability inequality satisfied by one system.
For instance, if $(\opA{M_0},B)$ is \EC in time $T_0$, then we can write
\begin{align*}
\norm{z}_H^2 &\leq C\int_0^{T_0} \norm{B^*\opS{M_0}(t)^*z}_U^2 \, dt
\\
& \leq 2C\left(
\int_0^{T_0} \norm{B^*\opS{M}(t)^*z}_U^2 \, dt
+\int_0^{T_0} \norm{B^*\opS{M_0}(t)^*z -B^*\opS{M}(t)^*z}_U^2 \, dt
\right).
\end{align*}
Thus, the operator
\begin{equation}\label{def op K}
Kz=B^*\opS{M_0}(\cdot)^*z -B^*\opS{M}(\cdot)^*z,
\end{equation}
is a natural candidate for inequality \eqref{obs with cpct rem}.
We can also start the same reasoning with the system $(\opA{M},B)$ instead.
A little subtlety is that the operator $K$ is a priori defined only on $D(\opA{}^*)$.
However, thanks to the admissibility inequality \eqref{admi ineq} it has a unique extension to $H$, that we shall still denote by $K$.
In summary, we have:

\begin{corollary}\label{cor DO18}
Assume that the operator $K:H \to L^2(0,T;U)$ defined by \eqref{def op K} is compact for all $T>0$.
Then, $\TinfEC(\opA{M},B)=\TinfEC(\opA{M_0},B)$.
\end{corollary}

It is in general not so easy to check that this $K$ is compact.
However, if the semigroup $(\opS{M_0}(t))_{t \geq 0}$ is explicitly known for some specific $M_0$, then some suitable criterion has been developed in \cite{HO21-JMPA}:

\begin{proposition}\label{prop crit K cpct}
Let $P=M-M_0$ and denote by the same the multiplication operator induced on $H$.
For $z \in H$ and $t \geq 0$, let
$$
(\Psi z)(t)=\int_0^t \opS{M_0}(t-s)^*P^*\opS{M_0}(s)^*z \, ds.
$$
Assume that there exists $\epsilon>0$ such that $\Psi:H \to L^2(0,\epsilon;H)$ and $B^*\Psi:H \to L^2(0,\epsilon;U)$ are compact.
Then, the operator $K:H \to L^2(0,T;U)$ defined by \eqref{def op K} is compact for all $T>0$.
\end{proposition}

We point out that the operator $B^*\Psi$ is a priori well defined only on $D(\opA{}^*)$.
However, we can show that (see e.g. \cite[Appendix A]{DO18}), for every $f \in C^1([0,T];H)$, we have $\int_0^t \opS{M_0}(t-s)^*f(s) \,ds \in D(\opA{}^*)$ for every $t \in [0,T]$, with
\begin{equation}\label{key estim}
\int_0^T \norm{B^* \int_0^t \opS{M_0}(t-s)^*f(s) \,ds}_{U}^2 \, dt
\leq C \norm{f}_{L^2(0,T;H)}^2,
\end{equation}
for some $C>0$ independent of $f$.
Therefore, the operator $B^*\Psi$ has a unique extension to $H$ and as usual we shall keep the same notation to denote this extension.

\begin{proof}[Proof of Proposition \ref{prop crit K cpct}]
The proof is inspired from the one of \cite[Lemma 3]{NRL86}, by taking care in addition of the operator $B^*$.
The starting point is the integral equation satisfied by semigroups of boundedly perturbed operators (see e.g. \cite[Corollary III.1.7]{EN00}):
$$
\opS{M}(t)^*z=
 \opS{M_0}(t)^*z
+\int_0^t \opS{M_0}(t-s)^*P^*\opS{M}(s)^*z \,ds,
$$
valid for every $t \geq 0$ and $z \in H$.
Thus, we see that $h(t)=\opS{M}(t)^*z-\opS{M_0}(t)^*z$ is the solution in $L^2(0,T;H)$ to the Volterra integral equation
\begin{equation}\label{Volt IE}
(\Id-V)h=\Psi z,
\end{equation}
where $(Vh)(t)=\int_0^t \opS{M_0}(t-s)^*P^* h(s) \,ds$.
It is well-known that $\Id-V$ is invertible (see for instance the proof of \cite[Lemma D.5]{HO21-JMPA}), so that
\begin{equation}\label{estim Volt}
\norm{h}_{L^2(0,T;H)} \leq C \norm{\Psi z}_{L^2(0,T;H)},
\end{equation}
for some $C>0$ independent of $z$.
Also, from equation \eqref{Volt IE} we get the identity
$$
Kz=B^*Vh +B^*\Psi z.
$$
It then follows from \eqref{key estim} and \eqref{estim Volt} that
$$
\norm{Kz}_{L^2(0,T;U)} \leq C' \left(
\norm{\Psi z}_{L^2(0,T;H)}
+\norm{B^*\Psi z}_{L^2(0,T;U)}
\right),
$$
for some $C'>0$ independent of $z$.
This shows that $K$ is compact if so are $\Psi$ and $B^*\Psi$.
Finally, thanks to basic semigroup properties, we can show that $\Psi:H \to L^2(0,T;H)$ is compact for all $T>0$ if it is compact for some $T>0$, and the same for $B^*\Psi$.
We refer for instance to the proof of \cite[Lemma 4.8]{HO21-JMPA} for details.

\end{proof}

\begin{remark}\label{rem Lebesgue}
Thanks to Lebesgue's dominated convergence theorem, the assumptions of Proposition \ref{prop crit K cpct} are satisfied if, for a.e. $t \in (0,\epsilon)$:
\begin{itemize}
\item
$\norm{(B^*\Psi z)(t)}_U \leq C \norm{z}_H$
for some $C>0$ independent of $t$ and $z$.

\item
$z \in H \mapsto (\Psi z)(t) \in H$ and $z \in H \mapsto (B^*\Psi z) (t) \in U$ are compact.
\end{itemize}
\end{remark}

\subsection{Compactness of some integral operators}\label{sect compact int op}

It remains to find the good candidate for $M_0$ in order to apply Proposition \ref{prop crit K cpct}.
Let $\varphi(t)=\opS{M_0}(t)^*z$.
It is the solution to the (forward) \imp{adjoint system}
$$
\begin{dcases}
\ddt \varphi(t)=\opA{M_0}^*\varphi(t), \quad t \in (0,T), \\
\varphi(0)=z.
\end{dcases}
$$
More explicitly,
\begin{equation}\label{forward adj syst}
\begin{dcases}
\pt{\varphi}(t,x)-\Lambda(x) \px{\varphi}(t,x)= \left(M_0(x)^\tr +\Lambda'(x)\right) \varphi(t,x), \\
\varphi_-(t,0)=\Qadj^\tr \varphi_+(t,0), \quad \varphi_+(t,1)=0, \\
\varphi(0,x)=z(x),
\end{dcases}
\quad (t,x) \in R_T,
\end{equation}
where we recall that $\Qadj$ is defined in \eqref{def Qadj}.
Therefore, the solution to \eqref{forward adj syst} is explicit and the simplest for
$$M_0=-\Lambda'.$$
Denoting by $\chir_\alpha$ the flow associated with $-\lambda_\alpha$, it is given by
$$
\varphi_i(t,x)=
\begin{dcases}
z_i(\chir_i(0;t,x)) & \text{ if } \chir_i(0;t,x) \in (0,1),
\\
0 & \text{ if } \chir_i(0;t,x) \not\in (0,1) \text{ and } i \geq m+1,
\\
(\Qadj^\tr)_i \varphi_+(\ssinr_i(t,x),0) & \text{ if } \chir_i(0;t,x) \not\in (0,1) \text{ and } i \leq m,
\end{dcases}
$$
where $(\Qadj^\tr)_i$ denotes the $i$-th row of $\Qadj^\tr$ and $\ssinr_i(t,x)$ is the solution to $\chir_i(\ssinr_i(t,x);t,x)=0$.

We can then deduce an explicit formula for $\Psi$ as well.
We are not going to detail all the tedious computations here (see \cite[Section 4.2]{HO21-JMPA}), we just want to show where the main assumption \eqref{hyp NRL86} is involved to prove the compactness.
We also refer to \cite{NRL86}.
Let $t$ be fixed.
Let us consider the case $i \geq m+1$ for simplicity.
Then, for a.e. $x \in (0,1)$, we have
\begin{align*}
(\Psi z)(t)_i(x) &=
\int_I  (P^*\opS{M_0}(s)^*z)_i(\xi(s)) \, ds
\\
&
=\sum_{j \geq m+1} \int_I  p_{ji}(\xi(s))
\varphi_j(s,\xi(s)) \, ds
+\sum_{j \leq m} \ldots
\\
&
=\sum_{j \geq m+1} \int_{\tilde{I}_j}  p_{ji}(\xi(s))
z_j(\tilde{\xi}_j(s)) \, ds
+ \ldots,
\end{align*}
where $\xi(s)=\chir_i(0;t-s,x)$, $I=\ens{s \in (0,t) \st \xi(s) \in (0,1)}$, $\tilde{\xi}_j(s)=\chir_j(0;s,\xi(s))$ and $\tilde{I}_j=\ens{s \in I \st \tilde{\xi}_j(s) \in (0,1)}$.
It is well-known that the above integral operators with $\tilde{\xi}_j(s)$ replaced by $s$ are compact on $L^2(0,1;\C)$.
Therefore, if we can do a change of variable, we will obtain a compact operator.
A direct computation shows that a uniform bound of the form
$$\abs{\tilde{\xi}_j'(s)} \geq C>0,$$
holds if, and only if, $\lambda_j(x) \neq \lambda_i(x)$ for all $x$.
In case $\lambda_j(x_0)=\lambda_i(x_0)$ for some $x_0$, we avoid possible problems if we have $p_{ji}(\xi)=0$ for a.e. $\xi \in (0,1)$.
This is exactly our main assumption \eqref{hyp NRL86}, but assuming in addition that the diagonal of $M$ is equal to $-\Lambda'$.
This restriction can however be removed since, up to changing $M$ to another internal coupling matrix $M^0$, the diagonal values can be made arbitrary without changing assumption \eqref{hyp NRL86}.
This is achieved with simple diagonal change of unknowns (the same that we will use in Section \ref{sect remove diag} below, but modifying appropriately the definition of $D$).

The compactness of $(B^*\Psi z)(t)$ can be proved similarly.
In conclusion, applying the abstract results of Section \ref{sect compact uniq}, we get
$$(\TinfEC(\opA{M},B)=) \, \TinfEC(\opA{M^0},B)=\TinfEC(\opA{-\Lambda'},B) \, (=\TinfEC(\opA{0},B)).$$
For more details we refer to \cite{HO21-JMPA}.

\section{Results for null controllability}

We have seen in the previous sections that the minimal control time is known for systems of conservation laws ($M=0$) and for the \EC of systems of balance laws ($M \neq 0$).
The situation is more intricate for the \NC of systems of balance laws.
This problem was first investigated in \cite{CN19} and followed by a series of works \cite{HO21-JMPA,CN21,HO21-COCV,CN21-pre,HO22-JDE,HO25-JDE,HO25-SICON}.
The goal of this part is to discuss some of these results.
All along this part, we will assume that
\begin{equation}\label{hyp strict}
\lambda_1<\ldots< \lambda_m<0<\lambda_{m+1}<\ldots< \lambda_{m+p}.
\end{equation}
In particular, we have again the order \eqref{order times}.

\subsection{The backstepping method}\label{sect BS}

In all the works mentioned above, the starting point in the investigation of the \NC of system \eqref{syst} is the \imp{backstepping method} for PDEs.
Popularized with the reference book \cite{KS08} (see also the more recent survey \cite{VABAK26}), this method was then developed specifically for hyperbolic systems in a series of works \cite{CVKB13,DMVK13,HDM15,HDMVK16,HVDMK19}.
Roughly speaking, it consists in using a Volterra integral transformation of the second kind to obtain an equivalent system (often called target system) for which the controllability properties are simpler to discuss.
Let us give some details.
The computations in this section are rather formal but they can be rigorously justified.

Below, we shall continue to use the notion of \imp{equivalent systems}, still denoted by $\sim$, in the same meaning as in Definition \ref{def syst equiv}.

\subsubsection{Removal of diagonal internal couplings}
\label{sect remove diag}

To use the backstepping method, we first need to remove the diagonal entries of $M$ (this will be clear from what follows).
This can easily be done by using the transformation
\begin{equation}\label{pre expo transfo}
y^*(t,x)=D(x)y(t,x),
\end{equation}
where $D(x)=\diag(d_1(x),\ldots,d_n(x))$ with $d_k(x)=e^{-\int_0^x \frac{m_{kk}(\xi)}{\lambda_k(\xi)} \,d\xi}$.
Indeed,
\begin{itemize}
\item
The spatial transformation \eqref{pre expo transfo} is invertible on the state space $L^2(0,1)^n$.

\item
$y^*$ and $y$ satisfy the same boundary condition at $x=0$ since $D(0)=\Id_n$.
Besides, if $y$ is a solution to the system $(\Lambda,Q,M)$, then
$$
\pt{y^*}(t,x)+\Lambda(x) \px{y^*}(t,x)=(D(x)M(x) +\Lambda(x)D'(x)) y(t,x).
$$
Thus, $y^*$ is a solution to the system $(\Lambda,Q,M^0)$, where
\begin{equation}\label{def M0}
M^0(x)=(D(x)M(x) +\Lambda(x)D'(x)) D(x)^{-1}.
\end{equation}
Conversely, we easily check that, if $y^*$ is a solution to $(\Lambda,Q,M^0)$, then $y$ defined by \eqref{pre expo transfo} is a solution to $(\Lambda,Q,M)$.

\end{itemize}

In other words, we have shown that
$$(\Lambda,Q,M) \sim (\Lambda,Q,M^0).$$

Note that $M^0$ has the same regularity as $M$ and that the diagonal entries of $M^0$ are all equal to zero by construction.
Note as well that $M^0=M$ if the diagonal entries of $M$ are already equal to zero.

\subsubsection{Volterra transformation}

The core of the backstepping method consists in using the Volterra integral transformation of the second kind
\begin{equation}\label{Volt transfo}
\tilde{y}(t,x)=y^*(t,x)-\int_0^x K(x,\xi)y^*(t,\xi) \, d\xi,
\end{equation}
where $K:\Tau \to \R^{n \times n}$ is a kernel to be determined, that is defined on the triangle $\Tau=\ens{(x,\xi) \st 0<\xi<x<1}$.

\begin{itemize}
\item
It is well-known that the spatial transformation \eqref{Volt transfo} is invertible on the state space $L^2(0,1)^n$ provided for instance that $K \in L^{\infty}(\Tau)^{n \times n}$ (see e.g. \cite[Chap. 2, Theorem 6]{Hoc73}).

\item
$\tilde{y}$ and $y^*$ clearly satisfy the same boundary condition at $x=0$.
Besides, if $y^*$ is a solution to the system $(\Lambda,Q,M^0)$, then, using the equation satisfied by $y^*$, integrating by parts, and using the boundary condition satisfied by $y^*$ at $x=0$, we obtain
\begin{align*}
\MoveEqLeft[2]
\pt{\tilde{y}}(t,x)+\Lambda(x) \px{\tilde{y}}(t,x)=
\\
&-\int_0^x \left(
\Lambda(x)\px{K}(x,\xi)
+\pxi{K}(x,\xi)\Lambda(\xi)
+K(x,\xi) (\Lambda'(\xi)+M^0(\xi))
\right)y^*(t,\xi) \, d\xi
\\
&+(M^0(x)+K(x,x)\Lambda(x)-\Lambda(x)K(x,x))y^*(t,x)
\\
&-K(x,0)\Lambda(0)\begin{pmatrix} \Id_m \\ Q \end{pmatrix} y^*_-(t,0).
\end{align*}
Therefore, if $K$ satisfies the following so-called kernel equations:
\begin{equation}\label{kern equ}
\begin{dcases}
\Lambda(x)\px{K}(x,\xi)
+\pxi{K}(x,\xi)\Lambda(\xi)
+K(x,\xi) (\Lambda'(\xi)+M^0(\xi))=0,
\\
\Lambda(x) K(x,x)-K(x,x)\Lambda(x) =M^0(x),
\end{dcases}
\quad (x,\xi) \in \Tau,
\end{equation}
then $\tilde{y}$ is the solution to
\begin{equation}\label{syst G}
\begin{dcases}
\pt{\tilde{y}}(t,x)+\Lambda(x) \px{\tilde{y}}(t,x)=G(x) \tilde{y}_-(t,0), \\
\tilde{y}_-(t,1)=\tilde{u}(t), \quad \tilde{y}_+(t,0)=Q\tilde{y}_-(t,0), \\
\tilde{y}(0,x)=\tilde{y}^0(x),
\end{dcases}
\end{equation}
for some initial data $\tilde{y}^0$ and function $\tilde{u}$, with $G:[0,1] \to \R^{n \times m}$ given by
\begin{equation}\label{def G}
G(x)
=
K(x,0)P, \quad P=-\Lambda(0)\begin{pmatrix} \Id_m \\ Q \end{pmatrix}.
\end{equation}

Conversely, we can check that, if $\tilde{y}$ is a solution to system \eqref{syst G}, then $y^*$ defined by \eqref{Volt transfo} is a solution to $(\Lambda,Q,M^0)$.

\end{itemize}

In other words, provided that a solution to the kernel equations \eqref{kern equ} actually exists, we have shown that
$$(\Lambda,Q,M^0) \sim (\Lambda,Q,G).$$
The notation $(\Lambda,Q,G)$ with ``$G$'' will be exclusively used to refer to system \eqref{syst G}, there will be no risk of confusion.

The main difficulty in the backstepping method is to prove the existence of a solution to the kernel equations.
Let us first point out that they have a solution only if the diagonal entries of $M^0$ are equal to zero (because of the condition at $(x,x)$ in \eqref{kern equ}).
This explains the need to change $M$ to $M^0$ by a preliminary transformation, as discussed before.
The standard approach to prove the existence of a solution, initiated in \cite{CVKB13} (see also \cite{VKC11}), consists in reformulating these equations as a system of integral equations and then proving existence (and uniqueness) of a solution to this system of integral equations by using the Banach fixed-point theorem and appropriate estimates.
To perform this reformulation, however, we need additional conditions for $K$, sometimes referred to as ``artificial conditions''.
For instance, for the diagonal entries of $K$ there is no condition other than the first equation in \eqref{kern equ}.
Usually, we consider an artificial condition at $(x,0)$ for these entries, since it is a point of the boundary of $\Tau$, and it also appears in the definition of $G$, which can be helpful to simplify its structure (see e.g. Section \ref{sect 1+1 Q not zero} below).
For $n=2$, which was the situation considered in \cite{CVKB13}, conditions at $(x,x)$ and $(x,0)$ are enough to carry out the aforementioned reformulation.
When $n \geq 3$, this is not enough though and we may need to add artificial conditions for some components of $K$, depending on the relation between the corresponding speeds in $x$ or $\xi$ direction.
This was first pointed out in \cite{HDM15} for $2+1$ systems.
Existence of a solution to the kernel equations for general $m+p$ systems was eventually proved in \cite{HDMVK16} (see also \cite{HVDMK19}).
Finally, let us also mention that a different approach was recently proposed in \cite{HO25-SICON}, by considering the condition at $(x,x)$ as an initial condition and finding the corresponding so-called domain of determinacy, much in the spirit of the reference books \cite{LY85,Bre00}.
This second approach is especially useful to prove the existence of regular solutions (see \cite[Remark A.5]{HO25-SICON}).

\begin{remark}
The kernel equations \eqref{kern equ} and the equivalent system \eqref{syst G} are not the only ones that can be obtained via a transformation of the form \eqref{Volt transfo}.
Other interesting choices are possible (see e.g. \cite{DMVK13,HDMVK16,ADM16}).
\end{remark}

\subsubsection{Control of the equivalent system}
\label{sect BS problems}

From the backstepping method described above we see that the study of the controllability properties of our initial system \eqref{syst} comes down to the study of the controllability properties of systems of the form \eqref{syst G}.
The gain here is the simpler coupling structure, thanks to the term $G(x) \tilde{y}_-(t,0)$, rather than $M(x)y(t,x)$.
Let us comment two approaches.

\paragraph{Direct approach.}
Let us write the block decomposition
\begin{equation}\label{G decomp}
G=
\begin{pmatrix}
G_{--} \\
G_{+-}
\end{pmatrix},
\quad \begin{array}{l}
G_{--} \in \R^{m \times m}, \\ G_{+-} \in \R^{p \times m}.
\end{array}
\end{equation}
Then, we see that $\tilde{u}$ directly controls the components of $\tilde{y}_-$:
$$
\begin{dcases}
\pt{\tilde{y}_-}(t,x)+\Lambda_-(x) \px{\tilde{y}_-}(t,x)=G_{--}(x)\tilde{y}_-(t,0), \\
\tilde{y}_-(t,1)=\tilde{u}(t), \\
\tilde{y}_-(0,x)=\tilde{y}^0_-(x),
\end{dcases}
$$
and indirectly controls the components of $\tilde{y}_+$ through $v(t)=\tilde{y}_-(t,0)$:
\begin{equation}\label{syst y+}
\begin{dcases}
\pt{\tilde{y}_+}(t,x)+\Lambda_+(x) \px{\tilde{y}_+}(t,x)=G_{+-}(x)v(t), \\
\tilde{y}_+(t,0)=Qv(t), \\
\tilde{y}_+(0,x)=\tilde{y}^0_+(x).
\end{dcases}
\end{equation}
The system is thus in what is often called ``cascade form''.
Note however, that the values of $v$ are not free to choose everywhere.
Indeed, by definition $\tilde{y}_-$ satisfies (we assume that the speeds are constant for simplicity)
\begin{equation}\label{explicit formula y-}
\tilde{y}_k(t,x)=
\begin{dcases}
\tilde{y}^0_k(-\lambda_k t+x)
+\int_0^t G_k(\lambda_k(s-t)+x) \tilde{y}_-(s,0) \, ds
& \mbox{ if } t+\frac{1-x}{\lambda_k}<0, \\
\tilde{u}_k\left(t+\frac{1-x}{\lambda_k}\right)
+\int_{t+\frac{1-x}{\lambda_k}}^t G_k(\lambda_k(s-t)+x) \tilde{y}_-(s,0) \, ds
& \mbox{ if } t+\frac{1-x}{\lambda_k}>0,
\end{dcases}
\end{equation}
for all $k \leq m$, where $G_k$ denotes the $k$-th row of $G$.
In particular, $v$ has to satisfy
\begin{equation}\label{cond vk on (0,Tk)}
v_k(t)=\tilde{y}^0_k(-\lambda_k t)
+\int_0^t G_k(\lambda_k(s-t)) v(s) \, ds, \quad t \in (0,T_k).
\end{equation}
After $T_k$, the values of $v_k$ are free to choose, because they relate to the control $\tilde{u}_k$.
Finally, for later purposes, we also write the explicit solution to system \eqref{syst y+}:
\begin{equation}\label{explicit formula y+}
\tilde{y}_{m+k}(t,x)=
\begin{dcases}
\tilde{y}^0_{m+k}(-\lambda_{m+k} t+x)
+\int_0^t G_{m+k}(\lambda_{m+k}(s-t)+x) v(s) \, ds
& \mbox{ if } t-\frac{x}{\lambda_{m+k}}<0, \\
Q_k v\left(t-\frac{x}{\lambda_{m+k}}\right)
+\int_{t-\frac{x}{\lambda_{m+k}}}^t G_{m+k}(\lambda_{m+k}(s-t)+x) v(s) \, ds
& \mbox{ if } t-\frac{x}{\lambda_{m+k}}>0,
\end{dcases}
\end{equation}
for all $1 \leq k \leq p$, where $Q_k$ denotes the $k$-th rows of $Q$.

\paragraph{Duality approach.}
The simplification in the coupling structure is also reflected at the level of its adjoint system, which is
$$
\begin{dcases}
\pt{\varphi}(t,x)+\Lambda(x) \px{\varphi}(t,x)=-\Lambda'(x) \varphi(t,x), \\
\varphi_-(t,0)=\Qadj^\tr \varphi_+(t,0) +\int_0^1 \Gadj(\xi)^\tr\varphi(t,\xi) \, d\xi, 
\quad \varphi_+(t,1)=0, \\
\varphi(T,x)=z(x),
\end{dcases}
\quad (t,x) \in R_T,
$$
where we recall that $\Qadj$ is defined in \eqref{def Qadj} and $\Gadj(\xi)=-G(\xi)\Lambda_-(0)^{-1}$ (recall \eqref{decomp Lambda}).
Observe that the equations of the adjoint system are completely decoupled, it only consists in transport equations.
The system is only coupled at the boundary $x=0$.

\paragraph{Questions.}
Now, two problems naturally arise:
\begin{enumerate}[(I)]
\item\label{prob 1}
Can we characterize the minimal time for the \NC of systems of the form \eqref{syst G} in function of $Q$ and $G$ ?

\item\label{prob 2}
If so, can we deduce the minimal time for the \NC of system \eqref{syst} in function of $Q$ and $M$, when $G$ takes the particular form \eqref{def G} with \eqref{kern equ} ?
\end{enumerate}

\begin{remark}
The choice of solution to the kernel equations does not affect the controllability properties because all the corresponding systems \eqref{syst G}-\eqref{def G} are equivalent.
\end{remark}

In the next sections we present answers to the above problems in some situations.

\subsection{The set and the control time of Coron and Nguyen}
\label{sect CN19}

In \cite{CN19}, the authors introduced a class of boundary coupling matrices and proved the \NC for generic internal coupling matrices of the corresponding system in some time $\TCN$, in general smaller than the control time $T_m+T_{m+1}$ presented in Section \ref{sect LR}.
The goal of this part is to present this result.
To this end, we introduce the following classes of boundary coupling matrices:
\begin{equation}\label{def setCN}
\setCN_h=
\ens{Q \in \R^{p \times m} \st
\begin{pmatrix}
q_{11} & \cdots & q_{1k} \\
\vdots &  & \vdots \\
q_{k1} & \cdots & q_{kk}
\end{pmatrix}
\text{ is invertible for all $k \leq h$}},
\end{equation}
for $1 \leq h \leq \rhomax=\min\ens{m,p}$.
Note that $\setCN_{\rhomax} \subset \ldots \subset \setCN_1$.
We also set $\setCN_0=\R^{p \times m}$.
An important feature of the sets $\setCN_h$ is that they are stable by small perturbations.
Note as well that $r_k=c_k=k$ for all $k \leq h$ if $Q \in \setCN_h$ (recall Definition \ref{def rkck can form}).

\begin{theorem}[\cite{CN19}]\label{thm CN19}
Assume \eqref{hyp strict}.
Let $\rhoCN=\min\ens{m-1,p}$ ($\in \ens{0,\ldots,\rhomax}$).
For every $Q \in \setCN_{\rhoCN}$, and every $C \in L^{\infty}(0,1)^{n \times n}$, there exists a discrete set $E \subset \R$ such that, for every $\gamma \not\in E$, system \eqref{syst} with $M=\gamma C$ is \NC in time $T=\TCN$, where
$$
\TCN=
\max\ens{
T_{m+1}+T_1, \ldots, T_{m+\rhomax}+T_{\rhomax}, \quad
T_m
}.
$$
\end{theorem}

We will present a slightly different proof than the original one of \cite{CN19}.
In particular, $G_{--}$ is not necessarily strictly lower triangular and the space $H$ that will be introduced below is the Cartesian product of $L^2$ spaces on the same time interval.
However, the key point to conclude remains the same.

\begin{proof}
We only prove the result in the case $m\leq p$ for the sake of presentation, the other case can be treated similarly.
Thus, we assume that $Q \in \setCN_{m-1}$ and we are going to show that the system is \NC in time
$$\TCN=\max_{1 \leq k \leq m} T_{m+k} +T_k,$$
for generic $M$.
We also assume that the speeds are constant for simplicity.
We recall that, from the results of Section \ref{sect BS}, it is equivalent to show the \NC of the equivalent system \eqref{syst G}.
We continue to use the notations of this section here.

\begin{enumerate}
\item
We first find conditions on $v=\tilde{y}_-(\cdot,0)$ such that, in addition to \eqref{cond vk on (0,Tk)}, the corresponding solution $\tilde{y}_+$ to system \eqref{syst y+} satisfies $\tilde{y}_+(\TCN,\cdot)=0$.
Once such a $v$ has been constructed, we simply solve the backward system
$$
\begin{dcases}
\pt{\tilde{y}_-}(t,x)+\Lambda_- \px{\tilde{y}_-}(t,x)=G_{--}(x)v(t), \\
\tilde{y}_-(t,0)=v(t), \\
\tilde{y}_-(\TCN,x)=0,
\end{dcases}
$$
and take as final control the trace $\tilde{u}(t)=\tilde{y}_-(t,1)$.
Note that $\tilde{y}_-(0,x)=\tilde{y}^0_-(x)$ by definition of $v_k$ on $(0,T_k)$ (see \eqref{cond vk on (0,Tk)}).

\item
Let $1 \leq j \leq p$ be arbitrary.
Since $\TCN \geq T_{m+j}$ (recall \eqref{order times}), the condition $\tilde{y}_{m+j}(\TCN,\cdot)=0$ is equivalent to (see \eqref{explicit formula y+} and do $t=\TCN-\frac{x}{\lambda_{m+j}}$)
\begin{equation}\label{cond vk on (T-Tmk,T-T2m)}
Q_j v(t) + \int_t^{\TCN} G_{m+j}(\lambda_{m+j}(s-t)) v(s) \, ds =0, \quad t \in (\TCN-T_{m+j},\TCN).
\end{equation}

\item
We will construct a solution $v$ satisfying the system of integral equations \eqref{cond vk on (T-Tmk,T-T2m)} for all $j$, by choosing the values of $v_k$ only in $(\TCN-T_{m+k},\TCN)$.
In the missing interval $(T_k,\TCN-T_{m+k})$ (recall condition \eqref{cond vk on (0,Tk)}), we will fix arbitrary values for $v_k$:
\begin{equation}\label{cond vk on (Tk,T-Tmk)}
v_k=\psi_k \text{ in } (T_k,\TCN-T_{m+k}),
\end{equation}
where $\psi_k$ is any given function of $L^2(T_k,\TCN-T_{m+k})$.

\item
For $j \geq m$, we solve the corresponding equations \eqref{cond vk on (T-Tmk,T-T2m)} by simply taking
\begin{equation}\label{cond vk on (T-T2m,T)}
v=0 \text{ in } (\TCN-T_{2m},\TCN).
\end{equation}
This is compatible with \eqref{cond vk on (0,Tk)} since $(0,T_k)$ and $(\TCN-T_{2m},\TCN)$ are disjoint intervals for all $k \leq m$.

\item
Gathering \eqref{cond vk on (0,Tk)}, \eqref{cond vk on (Tk,T-Tmk)}, \eqref{cond vk on (T-Tmk,T-T2m)} and \eqref{cond vk on (T-T2m,T)}, we thus want to find a solution $v$ in $H=L^2(0,\TCN)^m$ to
$$(J+K)v=f,$$
where the operator $J:H \to H$ is
$$
(Jv)_k(t)=
\begin{dcases}
v_k(t) & \text{ if } t \in (0,\TCN-T_{m+k}), \\
Q_kv(t) & \text{ if } t \in (\TCN-T_{m+k}, \TCN-T_{2m}), \\
v_k(t) & \text{ if } t \in (\TCN-T_{2m}, \TCN),
\end{dcases}
$$
the operator $K:H \to H$ is
$$
(Kv)_k(t)=
\begin{dcases}
-\int_0^t G_k(\lambda_k(s-t)) v(s) \, ds & \text{ if } t \in (0,T_k), \\
0 & \text{ if } t \in (T_k, \TCN-T_{m+k}), \\
\int_t^{\TCN} G_{m+k}(\lambda_k(s-t)) v(s) \, ds & \text{ if } t \in (\TCN-T_{m+k}, \TCN-T_{2m}), \\
0 & \text{ if } t \in (\TCN-T_{2m},\TCN),
\end{dcases}
$$
and $f \in H$ is
$$
f_k(t)=
\begin{dcases}
\tilde{y}^0_k(-\lambda_k t) & \text{ if } t \in (0,T_k), \\
\psi_k(t) & \text{ if } t \in (T_k, \TCN -T_{m+k}), \\
0 & \text{ if } t \in (\TCN -T_{m+k}, \TCN). \\
\end{dcases}
$$
\item
We first note that $J$ is invertible thanks to the assumption $Q \in \setCN_{m-1}$.
Indeed, let $h \in H$ be arbitrary.
We first solve $Jv=h$ in $(\TCN-T_{2m},\TCN)$.
This immediately gives $v=h$ therein.
The condition $(Jv)_m=h_m$ also gives that $v_m=h_m$ in $(0,\TCN-T_{2m})$ and thus $v_m=h_m$ everywhere.
Using the assumption $Q \in \setCN_{m-1}$ we can iterate this procedure to successively uniquely determine the values of $v$ in $(\TCN-T_{m+k},\TCN-T_{m+k+1})$ for $k=m-1,m-2,\ldots,1$.
Finally, it is clear that $Jv=h$ is equivalent to $v=h$ in $(0,\TCN-T_{m+1})$.

\item
Because of the simultaneous presence of integrals of the form $\int_0^t$ and $\int_t^{\TCN}$, the system $(J+K)v=f$ is a system of Fredholm integral equations.
The idea of \cite{CN19} is to solve it using the ``analytic Fredholm theorem''.
We now assume that $M$ is of the form $M=\gamma C$, with $\gamma \in \C$ and $C \in L^{\infty}(0,1)^{n \times n}$.
Let us denote by $K(\gamma)$ the corresponding operator $K$.
Note on the other hand that $J$ remains independent of $\gamma$.
It is well-known that $K(\gamma)$ is a compact operator.
It is not difficult to show that $K(\gamma)$ is small for small $\gamma$ so that $J+K(\gamma)$ is invertible for such $\gamma$.
Finally, it is shown in \cite[Lemma 3.3]{CN19} that $\gamma \in \C \mapsto K(\gamma) \in \lin{H}$ is analytic.
Then, by the analytic Fredholm theorem \cite[Theorem 8.92]{RR04}, there exists a discrete set $S \subset \C$ such that, for every $\gamma \not\in S$, the equation $(J+K(\gamma))v=f$ has a unique solution $v$.
This concludes the proof of the theorem.

\end{enumerate}

\end{proof}

\begin{remark}
It is written in \cite{CN19} that the control time $\TCN$ is ``optimal''.
However, similar to what we wrote in Remark \ref{rem Li sharp}, let us warn the reader that this does not mean that $\TCN$ is the minimal control time for $Q \in \setCN_{\rhoCN}$.
What is meant in \cite{CN19} is that, ``in general'', we cannot do better than $\TCN$ if $Q \in \setCN_{\rhoCN}$, where ``in general'' means ``for some $Q \in \setCN_{\rhoCN}$ and some $M$''.
\end{remark}

The assumption that $M$ has to be of the form $M=\gamma C$, where $\gamma$ is not really explicit, was later removed in \cite{CN21}, assuming the strict inequality $T>\TCN$.
We do not present the proof here because it is anyway contained in the results we should present below (see in particular Remark \ref{rem max Tinf for Q in setCN}).

\subsection{The minimal control time for $1+1$ systems}\label{sect Tinf 1+1}

One situation where the control problem has been completely solved is the case of $1+1$ systems (i.e. $m=p=1$).
In that case, $Q \in \R$ is just a scalar.
For $Q \neq 0$, the minimal control time can easily be deduced from the results in \cite{CVKB13}.
The case $Q=0$ was treated in \cite{HO21-COCV}.
The goal of this section is to discuss these two results.

\subsubsection{With boundary coupling}
\label{sect 1+1 Q not zero}

We start with the case $Q \neq 0$.
Let us first point out that this means that $Q$ is full row rank, and thus the minimal control time is $T_1+T_2$ by Corollary \ref{cor main thm EC}.
Let us however mention a simpler proof.
As mentioned in Section \ref{sect BS}, there are some freedom for the choice of the solution to the kernel equations \eqref{kern equ}.
As shown in \cite{CVKB13}, there exists a solution $K=(k_{\alpha\beta})_{1 \leq \alpha,\beta \leq 2}$ such that
\begin{equation}\label{BC k11}
k_{11}(x,0)=-\frac{1}{\lambda_1(0)} k_{12}(x,0) \lambda_2(0) Q, \quad x \in (0,1),
\end{equation}
and, since $Q \neq 0$, such that
$$k_{22}(x,0)=-\frac{1}{\lambda_2(0) Q} k_{21}(x,0) \lambda_1(0), \quad x \in (0,1).$$
These conditions respectively mean that $g_{11}=0$ and $g_{21}=0$ (see \eqref{def G}).
Thus, $G=0$ in the equivalent system \eqref{syst G}.
Clearly, the minimal control time of this system is $T_1+T_2$ for $Q \neq 0$.
Thus, so is the minimal control time of the initial system \eqref{syst}.

\subsubsection{Without boundary coupling}

Let us now consider the case $Q=0$.
This situation is more delicate and it needs the introduction of the following quantity:

\begin{definition}\label{def xf}
For $a>0$ and a function $f:(0,a) \longrightarrow \R$, we denote by $\Xm{a}{f} \in [0,a]$ the length of the largest interval of the form $(0,\ell)$ with $\ell \leq a$ where the function $f$ vanishes a.e.
\end{definition}

\begin{example}\label{ex length}
The function $f(x)=e^{-1/x}$ satisfies $\Xm{a}{f}=0$ for all $a$.
\end{example}

Then, the result is:

\begin{theorem}[\cite{HO21-COCV}]\label{main thm 1+1}
Assume that $m=p=1$ and $Q=0$.
Then, system \eqref{syst} is \NC in time $T$ if, and only if,
\begin{equation}\label{min time syst}
T \geq \max\ens{T_2, \quad \int_{\Xm{\xmax}{m_{21}}}^1 \left(\frac{1}{\abs{\lambda_1(\xi)}}+\frac{1}{\lambda_2(\xi)}\right) \, d\xi},
\end{equation}
where $\xmax \in (0,1)$ is the unique solution to $\int_0^{\xmax} \left( \frac{1}{\abs{\lambda_1(\xi)}}+ \frac{1}{\lambda_2(\xi)}\right) \, d\xi=T_2$.
\end{theorem}

\begin{example}
If $m_{21}(x)=e^{-1/x}$, then the minimal control time is $T_1+T_2$.
This also shows that no control does a better job than $u=0$ in that case.
\end{example}

\begin{remark}\label{rem cst speeds}
From formula \eqref{min time syst} we see that the minimal control time can be any time between $\max\ens{T_1,T_2}$ and $T_1+T_2$, depending on the values of $m_{21}$.
\end{remark}

As we shall see below, the difficult part of the result is the necessary condition.

Following the plan laid out in Section \ref{sect BS problems}, we first investigate Problem \ref{prob 1}.
As mentioned before, we can always find a solution to the kernel equations \eqref{kern equ} satisfying in addition condition \eqref{BC k11} so that $g_{11}=0$.
Thus, there is only one coupling remaining, namely $g_{21}$.
We have:

\begin{theorem}[\cite{HO21-COCV}]\label{thm can syst}
Assume that $m=p=1$ and $Q=g_{11}=0$.
Then, system \eqref{syst G} is \NC in time $T$ if, and only if,
\begin{equation}\label{min time can syst}
T \geq \max\ens{T_1+\int_{\Xm{1}{g_{21}}}^1 \frac{1}{\lambda_2(\xi)} \,d\xi, \quad T_2}.
\end{equation}

\end{theorem}

It is not difficult to check that $\tilde{u}=0$ is a control for system \eqref{syst G} in the above configuration.
It is the necessary condition that is difficult to establish.
First of all, we shall prove:

\begin{lemma}\label{lem lowest bound}
If the system is \NC in time $T$, then necessarily $T \geq \max \ens{T_1,T_2}$.
\end{lemma}

\begin{proof}
There is no need to assume either $Q=0$ or $g_{11}=0$ here.
We assume that the speeds are constant ($\lambda_k(x)=\lambda_k$) for simplicity.
Assume for instance that $T<T_2=1/\lambda_2$, the other case being similar.
Then, the open subset
$$
\omega=\ens{x \in (0,1) \quad \st \quad -\lambda_2 T+x \in (0,1)},
$$
is not empty.
By assumption, for any initial data $\tilde{y}^0$, there exists a controlled solution $\tilde{y}$.
For $x \in \omega$, the \NC condition $\tilde{y}_2(T,x)=0$ is equivalent to (see \eqref{explicit formula y+})
$$
0=
\tilde{y}^0_2(-\lambda_2 T+x)
+\int_0^T g_{21}(\lambda_2(s-T)+x) \tilde{y}_1(s,0) \, ds.
$$
Since $\tilde{y}^0_2 \in L^2(0,1)$ has to be arbitrary this shows that the bounded linear operator $K:L^2(0,T) \longrightarrow L^2(\omega)$ defined by
$$
(Kv)(x)
=-\int_0^T g_{21}(\lambda_2(s-T)+x) v(s) \, ds,
$$
is surjective.
This is impossible since its range is clearly a subset of $L^{\infty}(\omega)$.
\end{proof}

The rest of the proof of the theorem relies on the Titchmarsh convolution theorem (see e.g. \cite[Chapter XV]{Mik78}).

\begin{proof}[Proof of Theorem \ref{thm can syst} (necessity)]
The speeds are still assumed to be constant for simplicity.
Assume that $T$ does not satisfy \eqref{min time can syst} and let us show that the system cannot be \NC.
By Lemma \ref{lem lowest bound} we can also always assume that $T \geq T_2$.
Now, for $T \geq T_2$, we have $\tilde{y}_2(T,\cdot)=0$ if, and only if (see \eqref{explicit formula y+}),
$$
\int_{T-T_2 x}^T g_{21}(\lambda_2(s-T)+x) \tilde{y}_1(s,0) \, ds=0, \quad 0<x<1.
$$
Writing $x=\lambda_2(T-t)$, $\sigma=s-t$ and $\tau=T-t$, this is equivalent to
$$
\int_0^{\tau} \tilde{y}_1(-(\tau-\sigma)+T,0) g_{21}(\lambda_2\sigma) \, d\sigma=0,
\quad 0<\tau<T_2.
$$
This is a convolution.
Applying the Titchmarsh convolution theorem, we get
\begin{equation}\label{app Titch}
\Xm{T_2}{\tilde{y}_1(-\cdot+T,0)}+\Xm{T_2}{g_{21}(\lambda_2 \cdot)} \geq T_2.
\end{equation}
Since $g_{11}=0$, from the explicit expression of the solution \eqref{explicit formula y-}, we have
$$
\tilde{y}_1(-\theta+T,0)=
\begin{dcases}
\tilde{y}^0_1(\lambda_1(-\theta+T))
& \mbox{ if } \theta>T-T_1, \\
\tilde{u}(-\theta+T -T_1)
& \mbox{ if } \theta<T-T_1.
\end{dcases}
$$
Therefore, choosing $\tilde{y}_1^0$ so that
$$
\tilde{y}_1(-\theta+T,0) \neq 0, \quad \text{ for a.e. } \theta \in
\left(T-T_1, T-T_1+\epsilon\right),
$$
where $\epsilon>0$ is small enough, we have the upper bound
$$\Xm{T_2}{\tilde{y}_1(-\cdot+T,0)} \leq T-T_1.$$
Combing back to \eqref{app Titch}, we deduce that
$$
\Xm{T_2}{g_{21}(\lambda_2 \cdot)} \geq T_2-(T-T_1).
$$
This is clearly equivalent to $\Xm{1}{g_{21}} \geq (T_2 -(T-T_1))/T_2$, that is also $T \geq T_1 +(1-\Xm{1}{g_{21}})T_2$.

\end{proof}

Following the plan laid out in Section \ref{sect BS problems}, we now investigate Problem \ref{prob 2}.
To deduce Theorem \ref{main thm 1+1} from the previous result, it suffices to show the identity
\begin{equation}\label{from lG to lM}
T_1+\int_{\Xm{1}{g_{21}}}^1 \frac{1}{\lambda_2(\xi)} \,d\xi
=\int_{\Xm{\xmax}{m_{21}}}^1 \left(\frac{1}{\abs{\lambda_1(\xi)}}+\frac{1}{\lambda_2(\xi)}\right) \, d\xi,
\end{equation}
where $g_{21}$ is not anymore any function, but the one given by \eqref{def G}, that is $g_{21}(x,0)=-k_{21}(x,0)\lambda_1(0)$ in our special situation ($Q=0$), where $k_{21}$ is the solution to
$$
\begin{dcases}
\lambda_2(x)\px{k_{21}}(x,\xi)+\pxi{k_{21}}(x,\xi)\lambda_1(\xi)
+k_{21}(x,\xi) \lambda_1'(\xi)
+k_{22}(x,\xi)m^0_{21}(\xi)=0, \\
k_{21}(x,x)=\frac{m^0_{21}(x)}{\lambda_2(x)-\lambda_1(x)},
\end{dcases}
\quad (x,\xi) \in \Tau.
$$
Writing $k_{21}(\cdot,0)$ in function of $m_{21}^0$ with these equations, we can establish the identity
$$
\int_0^{\Xm{1}{g_{21}}} \frac{1}{\lambda_2(\xi)} \, d\xi
=
\int_0^{\Xm{\xmax}{m_{21}^0}} \left( \frac{1}{\abs{\lambda_1(\xi)}}+ \frac{1}{\lambda_2(\xi)}\right) \, d\xi,
$$
from which \eqref{from lG to lM} readily follows.
The details can be found in \cite[Section 4]{HO21-COCV}.

\begin{remark}\label{1+1 feedback}
Since $\tilde{u}=0$ is a control for the equivalent system \eqref{syst G}, plugging this into the formula \eqref{Volt transfo}, and using also \eqref{pre expo transfo}, we deduce that the control of the initial system \eqref{syst} is given in the form of a (bounded) \imp{feedback law} (i.e. of the form $u(t)=F(y(t,\cdot))$):
$$
u(t)=\int_0^1 \frac{k_{11}(1,\xi)d_1(\xi)}{d_1(1)} y_1(t,\xi)\,d\xi
+\int_0^1 \frac{k_{12}(1,\xi)d_2(\xi)}{d_1(1)} y_2(t,\xi)\,d\xi.
$$
\end{remark}

\begin{remark}
Theorem \ref{thm can syst} can easily be extended to the case $m=1$ and $p$ arbitrary (and $Q \in \R^p$ arbitrary), but no general result similar to Theorem \ref{main thm 1+1} is known in that situation, see \cite[Section 5]{HO21-COCV} for a discussion (see, however, \cite{HO25-SICON} when $M(x)=M$ is constant with $\diag M=0$).
\end{remark}

\subsection{Extremum values of the minimal control time}

Unlike \EC, the minimal time for the \NC depends, in general, on the internal coupling matrix $M$.
In \cite{HO22-JDE}, the authors found the smallest and largest values that this time can achieve with respect to $M$, when $\Lambda$ and $Q$ are fixed.

\begin{theorem}[\cite{HO22-JDE}]\label{main thm HO22-JDE}
Assume \eqref{hyp strict}.
We have
\begin{align}
\min_{M \in L^{\infty}} \, \TinfNC(\Lambda,Q,M) &= \TinfNC(\Lambda,Q,0),
\nonumber
\\
\max_{M \in L^{\infty}} \, \TinfNC(\Lambda,Q,M) &=
\max\ens{
T_{m+1}+T_{c_1}, \ldots, T_{m+\rho_0}+T_{c_{\rho_0}}, \quad
T_{m+\rho_0+1}+T_m
},
\label{RHS max Tinf}
\end{align}
where $\rho_0$ is the largest integer $i \in \ens{1,\ldots,p}$ such that $Q_1,\ldots,Q_i$ are linearly independent (and $\rho_0=0$ if $Q_1=0$).

\end{theorem}

Above and below, $Q_k$ denotes the $k$-th row of $Q$.
We recall that $\TinfNC(\Lambda,Q,0)$ is explicitly known (Theorem \ref{main thm syst cons laws}).
We also recall that the indices $c_k$ are defined in Definition \ref{def rkck can form}.
In the statement of the theorem above, we use the convention that the undefined quantities are simply not taken into account (cases $\rho_0=0$ and $\rho_0=p$).

\begin{remark}\label{rem max Tinf for Q in setCN}
Using the order \eqref{order times}, we see that $\max_M \, \TinfNC(\Lambda,Q,M)=\TCN$ if $Q \in \setCN_{\rhoCN}$ (see Section \ref{sect CN19} for the notations).
Thus, this generalizes the result of \cite{CN21} mentioned at the end of Section \ref{sect CN19}.
\end{remark}

By investigating the equality ``$\min=\max$'', we deduce the following condition (on $\Lambda$ and $Q$) for the minimal time for the \NC of system \eqref{syst} to be independent of the internal coupling matrix:

\begin{corollary}\label{cor HO22-JDE}
Assume \eqref{hyp strict} and $\rho_0=p$ or
$$
\max\ens{T_{m+1}+T_{c_1}, \ldots, T_{m+\rho_0}+T_{c_{\rho_0}}} \geq T_{m+\rho_0+1}+T_m,
$$
(with $0<\rho_0<p$).
Then, $\TinfNC(\Lambda,Q,M)=\TinfNC(\Lambda,Q,0)$.
\end{corollary}

Note in particular that we recover the second point of Corollary \ref{cor main thm EC} (recall that \NC and \EC are equivalent if $\rank Q=p$).

\begin{remark}\label{rem cor HO22}
Because of the order \eqref{order times}, the assumption of the previous corollary is for instance satisfied if $Q_1,\ldots,Q_{\rhomax}$ are linearly independent, where $\rhomax=\min\ens{m,p}$.
\end{remark}

\subsubsection{Smallest value of the minimal control time}

\begin{proof}[Sketch of proof of Theorem \ref{main thm HO22-JDE} (first identity)]
We assume that $(\Lambda,Q,M)$ is \NC in time $T$.
\begin{enumerate}
\item
From the results of Section \ref{sect BS}, there exists $G$ such that $(\Lambda,Q,G)$ is \NC in time $T$.

\item
Adapting the proof of Proposition \ref{thm equiv syst}, we get that there exists $\tilde{G}$ such that $(\Lambda,\Qcan,\tilde{G})$ is \NC in time $T$, where $\Qcan$ denotes the canonical form of $Q$.
We refer to \cite[Proposition 4.1]{HO22-JDE} for details.

\item
Mimicking the proof of Lemma \ref{lem lowest bound} we can show that, necessarily, $T \geq \max_i T_i$.
Suppose now that $\rho \neq 0$ (otherwise we are done) and that $T$ is such that
$$
\max_i T_i \leq T<T_{m+r_{k_0}}+T_{c_{k_0}},
$$
for some $k_0$.
Then, adapting the proof of Lemma \ref{lem lowest bound} and using that $\Qcan$ is in canonical form, we see that the \NC condition $y_{m+r_{k_0}}(T,x)=0$ for $x$ in the nonempty open subset
$$
\tilde{\omega}
=\ens{x \in (0,1) \quad \st \quad T-x T_{m+r_{k_0}} \in (0,T_{c_{k_0}})},
$$
is equivalent to the surjectivity of some operator from $L^2(0,T)^m$ to $L^2(\tilde{\omega})$ but whose range is clearly a subset of $L^{\infty}(\tilde{\omega})$.
We refer to \cite[Theorem 5.1]{HO22-JDE} for details.

\end{enumerate}

\end{proof}

\subsubsection{Largest value of the minimal control time}
\label{sect largest value}

\begin{proof}[Sketch of proof of Theorem \ref{main thm HO22-JDE} (second identity)]
Let us denote the right-hand side in \eqref{RHS max Tinf} by
$$
T_+=
\max\ens{
T_{m+1}+T_{c_1}, \ldots, T_{m+\rho_0}+T_{c_{\rho_0}}, \quad
T_{m+\rho_0+1}+T_m
}.
$$
We want to show that $(\Lambda,Q,M)$ is \NC in any time $T>T_+$.
\begin{enumerate}
\item
As in the beginning of the proof of the first identity of the theorem, it is sufficient to prove that all the systems $(\Lambda,\Qcan,G)$, $G$ arbitrary, are \NC in time $T$.

\item
Using a suitable Volterra integral transformation, we can show that it is enough to prove that all the systems $(\Lambda,\Qcan,G)$ are \NC in time $T$, for $G$ with the additional property
\begin{equation}\label{struct G+-}
g_{m+i,c_k}=0, \quad \forall k, \, \forall i \geq r_k.
\end{equation}
We refer to \cite[Proposition 6.1]{HO22-JDE} for details.

\item
Using a compactness-uniqueness argument, we can prove that the minimal time for the \NC of systems of the form $(\Lambda,Q,G)$ does not depend on $G_{--}$ (recall \eqref{G decomp}).
We refer to \cite[Theorem 7.1]{HO22-JDE} for details, the proof of which is inspired from the one of \cite[Theorem 2]{CN21} and \cite[Lemma 2.6]{DO18}.

\item
After much simplifications in the coupling structure, we can follow the characteristics to show that the systems $(\Lambda,\Qcan,G)$ with $G_{--}=0$ and \eqref{struct G+-} are \NC in time $T_+$.
We refer to \cite[Theorem 8.1]{HO22-JDE} for details.

\item
Finally, we know that $T_+=\TinfNC(\Lambda,Q,0)$ if the conditions of Corollary \ref{cor HO22-JDE} are satisfied (Theorem \ref{main thm syst cons laws}).
On the other hand, we can show that $T_{m+\rho_0+1}+T_m=\TinfNC(\Lambda,Q,M)$ for the matrix $M$ defined by
$$m_{m+i,m}=\frac{\lambda_{m+i}-\lambda_m}{-\lambda_m} \ell^{i,\rho_0+1}, \quad \forall i \geq \rho_0+1,$$
and zero otherwise, where $\ell^{\alpha,\beta}$ denotes the $(\alpha,\beta)$ entry of the inverse of a matrix $L$ arising the decomposition $LQU=\Qcan$.
We refer to \cite[Section 8.2]{HO22-JDE} for details.
\end{enumerate}

\end{proof}

\begin{remark}
It is not possible to show that the minimal time for systems of the form $(\Lambda,Q,G)$ does not depend on $G_{--}$ by means of a ``transformation approach'' (e.g. backstepping).
A counterexample is detailed in \cite[Appendix A]{HO22-JDE}.
\end{remark}

\subsection{The rank increasing method}

We conclude the overview of the results about the \NC of system \eqref{syst} with the presentation of a method that was recently introduced in \cite{HO25-JDE}.
All along this section, we assume that $\Lambda$ and $M$ are regular enough on $[0,1]$ (i.e. $C^r$ with $r$ as large as need be).

First of all, we recall that, from the results of Section \ref{sect BS}, we have
$$(\Lambda,Q,M) \sim (\Lambda,Q,G),$$
where $G$ is defined by \eqref{def G}.
Since $\Lambda$ and $M$ are regular, we can find a solution $K$ to the kernel equations \eqref{kern equ} that is also regular (see e.g. \cite{HVDMK19, HO25-SICON}).
Thus, $G$ is regular as well.
On the other hand, as already mentioned in Section \ref{sect largest value}, we can prove that $\TinfNC(\Lambda,Q,G)$ does not depend on $G_{--}$ (recall \eqref{G decomp}).
Systems $(\Lambda,Q,G)$ with $G_{--}=0$ will be denoted by
$$(\Lambda,Q,G_{+-}).$$
To be consistent with the notations of the previous sections, we denote the $k$-th row of $G_{+-}$ by $G_{m+k}$.
We first describe the method to determine $\TinfNC$ for such systems and we will then explain how it can be used to deduce $\TinfNC$ for system \eqref{syst}.

\subsubsection{The method for systems of the form $(\Lambda,Q,G_{+-})$}

Let $\rhomax=\min\ens{m,p}$.
The main idea is to construct systems $(\Lambda,Q^{[k]}, G^{[k]}_{+-})$, $k=1,2,\ldots,\rhomax$, such that:
\begin{itemize}
\item
$(\Lambda,Q^{[k]}, G^{[k]}_{+-})$ and $(\Lambda,Q^{[k-1]}, G^{[k-1]}_{+-})$ have the same minimal control time (denoting by $(\Lambda,Q^{[0]}, G^{[0]}_{+-})=(\Lambda,Q, G_{+-})$).

\item
The first $k$ rows of $Q^{[k]}$ are linearly independent.
\end{itemize}
If we succeed, then it essentially follows from Remark \ref{rem cor HO22} that
$$
\TinfNC(\Lambda,Q,G_{+-})=\TinfNC(\Lambda,Q^{[\rhomax]}, G^{[\rhomax]}_{+-})=\TinfNC(\Lambda,Q^{[\rhomax]}, 0),
$$
so that, by Theorem \ref{main thm syst cons laws},
$$
\TinfNC(\Lambda,Q,M)=\max \ens{T_m, T_{Q^{[\rhomax]}}},
\quad T_{Q^{[\rhomax]}}=\max_{k \leq \rhomax} \, T_{m+k}+T_{c_k},
$$
where the indices $c_k$ refer to $Q^{[\rhomax]}$.

The construction of the systems $(\Lambda,Q^{[k]}, G^{[k]}_{+-})$ is as follows.
We assume that $(\Lambda,Q^{[k-1]}, G^{[k-1]}_{+-})$ has been constructed.

\begin{itemize}
\item
We check if the $k$-th row of $Q^{[k-1]}$ is linearly independent with its first $k-1$ rows.
If that is the case, we just set $Q^{[k]}=Q^{[k-1]}$ and $G_{+-}^{[k]}=G_{+-}^{[k-1]}$.

\item
Otherwise, this means that there exist scalars $a_1,\ldots,a_{k-1} \in \R$ such that
$$Q^{[k-1]}_k -\sum_{l<k} a_l Q^{[k-1]}_l=0,$$
and we construct a new system $(\Lambda,\hat{Q}^{[k-1]}, \hat{G}^{[k-1]}_{+-})$, with same minimal control time, in two steps as follows:

\begin{itemize}
\item
\textbf{Step 1: we put the $k$-th row to zero.}
Adapting the proof of Proposition \ref{thm equiv syst}, we can show that
$$
(\Lambda,Q^{[k-1]}, G^{[k-1]}_{+-}) \sim (\Lambda,\tilde{Q}^{[k-1]}, \tilde{G}^{[k-1]}_{+-}),
$$
where $\tilde{Q}^{[k-1]}$ and $\tilde{G}^{[k-1]}_{+-}$ have the same $i$-th rows as $Q^{[k-1]}$ and $G^{[k-1]}_{+-}$ for $i \neq k$, and
$$
\tilde{Q}^{[k-1]}_k=0, \quad \tilde{G}^{[k-1]}_{m+k}(x)=G_{m+k}(x)- \sum_{l<k} a_l G_{m+l}\left(\frac{\lambda_{m+l}}{\lambda_{m+k}} x\right),
$$
(we assume that the speeds are constant for simplicity).

\item
\textbf{Step 2: we take the derivative of the $k$-th row.}
We obtain that
$$
\TinfNC(\Lambda,\tilde{Q}^{[k-1]}, \tilde{G}^{[k-1]}_{+-})=\TinfNC(\Lambda,\hat{Q}^{[k-1]}, \hat{G}^{[k-1]}_{+-}),
$$
where $\hat{Q}^{[k-1]}$ and $\hat{G}^{[k-1]}_{+-}$ have the same $i$-th rows as $\tilde{Q}^{[k-1]}$ and $\tilde{G}^{[k-1]}_{+-}$ for $i \neq k$, and
$$
\hat{Q}^{[k-1]}_k=G_{m+k}(0), \quad \hat{G}^{[k-1]}_{m+k}(x)=\lambda_{m+k}(x) G_{m+k}'(x).
$$
We refer to \cite[Section 2.3]{HO25-JDE} for details.

\end{itemize}

We have thus produced a new $k$-th row for the boundary coupling matrix.
We then repeat the previous discussion with the new system $(\Lambda,\hat{Q}^{[k-1]}, \hat{G}^{[k-1]}_{+-})$ instead of $(\Lambda,Q^{[k-1]}, G^{[k-1]}_{+-})$.
Now comes our \imp{main assumption}: we assume that, after a finite number of iterations, we eventually end up in the first case, that is the $k$-th row of the new boundary coupling matrix is linearly independent with its $k-1$ first rows.

\end{itemize}

\subsubsection{Link with the initial system}

For the previous method to be effective for our initial system \eqref{syst} as well, we see that we need to know the derivatives of $G_{+-}$ at $0$.
Recall that $G(x)=K(x,0)P$ by definition (see \eqref{def G}).
Therefore,
$$
G^{(N)}(0)=\frac{\partial^N K}{\partial x^N}(0,0)P,
$$
for all $N \in \N$.
But the derivatives of $K$ at $(0,0)$ can be computed from the derivatives of $M^0$ at $0$ thanks to the kernel equations \eqref{kern equ}.
Indeed, let $K=(k_{\alpha\beta})_{1 \leq \alpha,\beta \leq n}$ be a solution to the kernel equations \eqref{kern equ} which is smooth enough and denote by
$$
f_{\alpha\alpha}(x)=k_{\alpha\alpha}(x,0).
$$
Then, assuming that the speeds are constant for simplicity, the kernel equation for the $(\alpha,\beta)$-th component is
\begin{equation}\label{kern equ bis}
\begin{dcases}
\lambda_{\alpha}\px{k_{\alpha\beta}}(x,\xi)
+\pxi{k_{\alpha\beta}}(x,\xi)\lambda_{\beta}
=- \sum_{\gamma=1}^n k_{\alpha \gamma}(x,\xi) m^0_{\gamma \beta}(\xi),
\\
k_{\alpha\beta}(x,x)=f_{\alpha\beta}(x) \quad (\beta \neq \alpha), \quad k_{\alpha\alpha}(x,0)=f_{\alpha\alpha}(x),
\end{dcases}
\end{equation}
where we introduced the known quantity
$$
f_{\alpha\beta}(x)=
\frac{m^0_{\alpha \beta}(x)}{\lambda_{\alpha}-\lambda_{\beta}} \quad (\beta \neq \alpha).
$$
The derivatives of $k_{\alpha1},\ldots, k_{\alpha n}$ at $(0,0)$ can then be computed successively as follows:

\begin{itemize}
\item
\textbf{$0$-th order derivatives.}
From the condition at $(x,x)$ in \eqref{kern equ bis}, we obviously have
$$k_{\alpha\beta}(0,0)=f_{\alpha\beta}(0).$$

\item
\textbf{$1$-st order derivatives.}
Taking the derivative of the condition at $(x,x)$ in \eqref{kern equ bis}, we obtain the $2 \times 2$ system
\begin{equation}\label{deriv kern 1}
J^{(1)}_{\alpha\beta}
\begin{pmatrix}
\px{k_{\alpha\beta}}(0,0) \\
\pxi{k_{\alpha\beta}}(0,0)
\end{pmatrix}
=
\begin{pmatrix}
-\sum_{\gamma=1}^n k_{\alpha\gamma}(0,0)m^0_{\gamma\beta}(0) \\
f_{\alpha\beta}'(0)
\end{pmatrix},
\end{equation}
with
$$
J^{(1)}_{\alpha\beta}=
\begin{pmatrix}
\lambda_\alpha & \lambda_\beta \\
1 & 1
\end{pmatrix}
\quad (\beta \neq \alpha),
\quad
J^{(1)}_{\alpha\alpha}=
\begin{pmatrix}
\lambda_\alpha & \lambda_\alpha \\
1 & 0
\end{pmatrix}.
$$
The right-hand side in \eqref{deriv kern 1} is a known quantity from the previous computations.
Besides, the matrix $J^{(1)}_{\alpha\beta}$ is invertible (recall \eqref{hyp strict}) and its inverse is easy to compute.
Thus, we know all the first order partial derivatives of $k_{\alpha1},\ldots, k_{\alpha n}$ at $(0,0)$.

\item
\textbf{$2$-nd order derivatives.}
Taking twice the derivative of the condition at $(x,x)$ in \eqref{kern equ bis}, one time the derivative of the first equation with respect to $x$ and one time with respect to $\xi$, we obtain a $3 \times 3$ system with similar properties than as in the previous step.
We can then go on.

\end{itemize}
The general formula can be found in \cite[Theorem 4.1]{HO25-JDE}.

Last, it is clear that the derivatives of $M^0$ can be computed from those of $M$ (see \eqref{def M0}).
Several applications of this method are given in \cite[Section 5]{HO25-JDE}.

\section{Summary and open problems}

We conclude this survey with a summary and some open problems.

\subsection{Summary}

We have reviewed the classical and current developments on boundary controllability  of 1D linear hyperbolic balance laws of the form \eqref{syst}, with a particular emphasis on the characterization of the minimal control time.
As illustrated in the previous sections, this control problem is highly nontrivial for such strongly coupled systems.
Novel phenomena have been discovered and various methods have been developed in the search for the characterization of the minimal control time.
We have added Table \ref{table summary} for a concise summary of these results (assuming \eqref{hyp strict} for simplicity).

\begin{center}
\def\arraystretch{2}
\footnotesize
\begin{tabular}{|c|c|c|}
\hline
Assumptions & Results & References
\\
\hline
$\emptyset$ (resp. $\rank Q=p$) & \NC (resp. $\EC$)  in time $T_m+T_{m+1}$ & \cite{Rus78,Li10}
\\
\hline
$M=0$ & Formula for $\TinfNC,\TinfEC$ & \cite{Wec82,HO22-JDE}
\\
\hline
$Q \in \setCN_{\rhoCN}$, $M=\gamma C$, $\gamma \not\in E$ & \NC in time $T=\TCN$ & \cite{CN19}
\\
\hline
$\emptyset$ & $\TinfEC$ is the same as for $M=0$ & \cite{HO21-JMPA}
\\
\hline
$Q \in \setCN_{\rhoCN}$ & \NC in any time $T>\TCN$ & \cite{CN21}
\\
\hline
$m=p=1$ & Formula for $\TinfEC,\TinfNC$ & \cite{CVKB13,HO21-COCV}
\\
\hline
$\emptyset$ & Formula for $\min_M \TinfNC$ and $\max_M \TinfNC$ & \cite{HO22-JDE}
\\
\hline
$m=1$, $M(x)=M$, $\diag M=0$ & Formula for $\TinfNC$ & \cite{HO25-SICON}
\\
\hline
$\Lambda,M$ smooth enough & Method to compute $\TinfNC$ & \cite{HO25-JDE}
\\
\hline
\end{tabular}
\captionof{table}{\label{table summary} Brief summary on the controllability of system \eqref{syst}.}
\end{center}

\subsection{Related open problems}

There are still some situations where the minimal time for the \NC of system \eqref{syst} remains unknown (Table \ref{table summary} does not cover all situations).
This is obviously a possible future direction of research.
Other problems related to the controllability of system \eqref{syst} that we believe to be both interesting and challenging are the following:

\begin{enumerate}
\item
\textbf{Stabilization.}
We have seen at several occasions in the survey that, sometimes, the control for the \NC can be obtained in the form of a \imp{feedback law} $u(t)=F(y(t,\cdot))$ for some suitable operator $F$ (Remarks \ref{M=0 feedback} and \ref{1+1 feedback}).
We then speak of \imp{finite-time stabilization}.
Using the backstepping method, it has also been shown that system \eqref{syst} is stabilizable in finite time $T_m+T_{m+1}$ (\cite{CVKB13,HDMVK16,ADM16,CHO17}).
Is the minimal time for stabilization always the same as the minimal time for \NC ?
For instance, if $\rank Q=p$, can we prove that the minimal time for stabilization does not depend on $M$ ?

\item
\textbf{Time-dependent systems.}
When the coupling parameters $\Lambda,Q,M$ also depend on the time $t$, there are only a few results in the literature (see e.g. \cite{Wan06,CHOS21,CN21-pre,MAK22}).
Moreover, the situation can be drastically different.
Indeed, in \cite{CN21-pre} the authors showed for instance that the control time $T_m+T_{m+1}$ cannot be improved for some (smooth) time-dependent internal coupling matrix $M=M(t,x)$, even if $Q \in \setCN_{\rhoCN}$ (compare with Theorem \ref{thm CN19}).
It is also shown therein that this ``bad'' situation does not occur for time-analytic internal coupling matrices.
It would be interesting to find the minimal control time in that framework.

\item
\textbf{More underactuated systems.}
Another very interesting question is whether we can reduce the number of controls at the boundary.
Consider, instead of the (local) boundary conditions in \eqref{syst}, the more general (and nonlocal) boundary conditions of the form
$$
\begin{pmatrix}
y_-(t,1) \\
y_+(t,0)
\end{pmatrix}
=Q
\begin{pmatrix}
y_-(t,0) \\
y_+(t,1)
\end{pmatrix}
+J u(t),
$$
where now $Q \in \R^{n \times n}$, $J \in \R^{n \times m}$ and $u:(0,T) \to \R^m$ is still the control.
In the recent work \cite{CFMS24}, the authors proved, under some assumptions, that the corresponding system is \EC in time $T=T_1+\ldots+T_n$ when $M=0$.
What is the minimal control time ? What happens when $M \neq 0$ ?

\item
\textbf{Quasilinear systems.}
As mentioned in Remark \ref{rem other app of LR}, the method of Li and Rao can also be used in the context of nonlinear systems of the form
$$
\begin{dcases}
\pt{y}(t,x)+\Lambda(y(t,x)) \px{y}(t,x)=M(y(t,x)), \\
y_-(t,1)=u(t), \quad y_+(t,0)=Q(y_-(t,0)), \\
y(0,x)=y^0(x),
\end{dcases}
\quad (t,x) \in R_T.
$$
Notably, this method gives \NC of this system, locally near $y=0$ (assuming $M(0)=Q(0)=0$), in time $T>T_m+T_{m+1}$, where $T_k=1/\abs{\lambda_k(0)}$.
Here, we have to work in the framework of semi-global $C^1$ solutions (i.e. $T$ can be large, but the initial data has to be taken small in consequence, see e.g. \cite{Li10}).
When $M=0$, it was already shown that a smaller control time is possible, under appropriate assumptions on $Q$.
First in \cite{Hu15}, where it was proved that the system is \EC locally near $y=0$ in time $T>\max\ens{T_{m+1}+T_p, T_m}$ if $m \geq p$ and there exists $\bar{Q}$ such that $y_+=Q(y_-) \Longleftrightarrow (y_1,\ldots,y_p)=\bar{Q}(y_{p+1},\ldots,y_m,y_+)$ for all $y \in \R^n$ small enough.
Then in \cite{CN20}, where the authors extended their previous linear result (see Theorem \ref{thm CN19} with $C=0$) in this context, replacing their assumption on the boundary coupling matrix by the same for its linearization at $0$, i.e. $Q'(0) \in \setCN_{\rhoCN}$.
It would be interesting to see if these assumptions on $Q$ can be weakened.
On the other hand, the case $M \neq 0$ seems completely open.

\item
\textbf{Internal controllability.}
All along this survey, we discussed the controllability of the system from the boundary.
What happens if the controls act inside the equations instead ?
More precisely, consider the equations
$$
\pt{y}(t,x)+\Lambda(x) \px{y}(t,x)=M(x) y(t,x) +1_{\omega}(x) J u(t,x), \quad (t,x) \in R_T,
$$
where the control $u:R_T \to \R^m$ is now acting inside the equations, through the matrix $J \in \R^{n \times m}$, and has to be localized in space in a given open subset $\omega \subset (0,1)$ ($1_\omega$ denotes the characteristic function of this set).
There are very few results about the internal controllability of 1D first-order hyperbolic systems.
The first ones seem \cite{ZLR16} and \cite{ABCO17}.
More recently, in \cite{LLQ24}, the authors proved that the system, with local boundary conditions, with same number of controls than state variables ($J=\Id_n$) and with $\omega=(a,b)$ an interval,  is \EC in any time $T>\max\ens{a,1-b}(T_m+T_{m+1})$, provided that the boundary coupling matrices are invertible.
This result was generalized in \cite{HO24-COCV}, where the authors found the minimal control time, considered an arbitrary open set $\omega$ and also showed the necessity of the invertibility of the boundary coupling matrices.
What happens for underactuated systems ($\rank J<n$) in general ? What happens if we consider nonlocal boundary conditions ? What happens if $\omega=\omega(t)$ depends on $t$ (``moving control'') ?

\end{enumerate}

\section*{Acknowledgments}

The authors would like to thank Jean-Michel Coron for his unfailing support and for all he has done for us.
He is an outstanding researcher and a wonderful human being.

The second author would like to thank Shandong University for its hospitality.
This work was initiated while he was visiting there.
The first author would like to thank Jagiellonian University for its hospitality.
This work was continued while he was visiting there.
This project was supported by National Key R\&D Program of China (No. 2023YFA1009200), National Natural Science Foundation of China (No. 12471421) and National Science Centre, Poland UMO-2020/39/D/ST1/01136 and UMO-2023/50/E/ST1/00081.

\renewcommand*{\bibfont}{\small}
\printbibliography


\end{document}